\documentclass[12pt]{amsart}
\usepackage[colorlinks=true,pagebackref,hyperindex,citecolor=green,linkcolor=red]{hyperref}
\usepackage{amsmath}
\usepackage{amsfonts}
\usepackage{amssymb}
\usepackage{color}
\usepackage[all]{xy}  
\usepackage{enumerate}
\usepackage[top=1in, bottom=1in, left=1in, right=1in]{geometry}
\usepackage{mathrsfs}

\usepackage{stmaryrd}
\usepackage{hyperref}

\theoremstyle{definition}
\newtheorem{theorem}{Theorem}[section]

\newtheorem{question}[theorem]{Question}
\newtheorem{corollary}[theorem]{Corollary}
\newtheorem{lemma}[theorem]{Lemma}
\newtheorem{proposition}[theorem]{Proposition}

\newtheorem*{theorem*}{Theorem}
\theoremstyle{definition}
\newtheorem{definition}[theorem]{Definition}

\newtheorem{example}[theorem]{Example}

\newtheorem{remark}[theorem]{Remark}
\numberwithin{equation}{subsection}

\newcommand{\m}{\mathfrak{m}}
\newcommand{\n}{\mathfrak{n}}

\newcommand{\NN}{\mathbb{N}}

\newcommand{\QQ}{\mathbb{Q}}

\newcommand{\ee}{\operatorname{e}}
\newcommand{\pd}{\operatorname{pd}}

\newcommand{\Spec}{\operatorname{Spec}}
\newcommand{\Depth}{\operatorname{depth}}

\newcommand{\Hom}{\operatorname{Hom}}
\newcommand{\Ext}{\operatorname{Ext}}
\newcommand{\Tor}{\operatorname{Tor}}

\newcommand{\Supp}{\operatorname{Supp}}
\newcommand{\Ker}{\operatorname{Ker}}

\newcommand{\CoKer}{\operatorname{Coker}}

\newcommand{\Ann}{\operatorname{Ann}}
\newcommand{\Ass}{\operatorname{Ass}}	
\newcommand{\IM}{\operatorname{Im}}

\newcommand{\rk}{\operatorname{rk}}

\newcommand{\ls}{\leqslant}
\newcommand{\gs}{\geqslant}
\newcommand{\ds}{\displaystyle}
\newcommand{\eee}[1]{\underline{}^e \hspace{-0.05cm}{#1}}
\newcommand{\eeep}[1]{\underline{}^{e+e'} \hspace{-0.05cm}{#1}}
\newcommand{\p}{\mathfrak{p}}
\newcommand{\q}{\mathfrak{q}}
\newcommand{\MIN}{\operatorname{Min}}
\newcommand{\ov}[1]{\overline{#1}}
\newcommand{\ann}{\operatorname{ann}}
\newcommand{\ps}[1]{\llbracket {#1} \rrbracket}

\begin{document}
\newcommand{\tens}{\otimes}
\newcommand{\hhtest}[1]{\tau ( #1 )}
\renewcommand{\hom}[3]{\operatorname{Hom}_{#1} ( #2, #3 )}

\title{Frobenius Betti numbers and modules of finite projective dimension}
\author{Alessandro De Stefani}

\author{Craig Huneke}

\author{Luis N\'u\~nez-Betancourt}

\begin{abstract}
Let $(R,\m,K)$ be a local ring, and  let $M$ be an $R$-module of finite length. We study asymptotic invariants, $\beta^F_i(M,R),$ defined by twisting with Frobenius the free resolution of $M$.
This family of invariants includes the Hilbert-Kunz multiplicity ($e_{HK}(\m,R)=\beta^F_0(K,R)$). 
We discuss several  properties of these numbers that resemble the behavior of the Hilbert-Kunz multiplicity.  
Furthermore, we study when the vanishing of  $\beta^F_i(M,R)$ implies that $M$ has finite projective dimension. In particular, we give a complete characterization of the vanishing of $\beta^F_i(M,R)$ for one-dimensional rings. 
As a consequence of our methods, we give conditions for the non-existence of syzygies of finite length.
\end{abstract}

\maketitle

{\hypersetup{linkcolor=black}

\tableofcontents
}

\section{Introduction}

Let $(R,\m,K)$ denote an $F$-finite local ring of dimension $d$ and characteristic $p>0,$ and let  $\alpha=\log_p[K:K^p].$ Given an $R$-module $M$ and an integer $e\gs 0$, $^e M$ denotes the $R$-module structure on $M$ given by $r *m=r^{p^e}m$ for every $m\in \eee M$ and $r \in R$. In addition, $\lambda_R(M)$, or simply $\lambda(M)$ when the ring is clear from the context, denotes the length of $M$ as an $R$-module.

Let $q=p^e$ be a power of $p$. For an ideal $I\subseteq R$, let $I^{[q]} = (i^q \mid i \in I)$ be the ideal generated by the $q$-th powers of elements in $I$. If $I$ is $\m$-primary, the \emph{Hilbert-Kunz multiplicity of $I$ in $R$} is defined by
$$e_{HK}(I,R)=\lim\limits_{e\to \infty} 
\frac{\lambda(R/I^{[q]})}{q^d}.
$$
The existence of the  previous limit was proven by Monsky  \cite{MonskyHK}.
Under mild conditions,
$e_{HK}(\m,R)=1$ if and only if $R$ is a regular ring \cite{WY}. 
The Hilbert-Kunz multiplicity can be interpreted as a measure of singularity: the smaller it is, the nicer the ring is. For instance, Aberbach and Enescu proved rings with small Hilbert-Kunz multiplicity are Gorenstein and $F$-regular \cite{AE-HK} (see also \cite{BE-HK}). We have that
$$
\lambda(R/I^{[q]})=q^\alpha\lambda(R/I \otimes_R {^e}R)=q^\alpha\lambda(\Tor^R_{0} (R/I, {^e}R))
$$
This gives rise to the following extension of the Hilbert-Kunz multiplicity for higher Tor functors. Let $N$ be a finitely generated $R$-module. For an integer $i \gs 0$ define
\[
\beta^F_i(M,N)=\lim\limits_{e\to \infty} \frac{\lambda(\Tor^R_{i} (M, {^e}N))}{q^{(d+\alpha)}}.
\]
We denote $\beta^F_i(K,R)$ by $\beta^F_i(R)$  and call it the {\it $i$-th Frobenius Betti number of $R$}.

These higher invariants also detect regularity. Namely,  Aberbach and Li \cite{AberbachLi} show that  $R$ is a regular ring
if and only if $\beta_i^F(R)=0$ for some $i\gs 1$.
Note that $R$ is regular if and only if $K$ has finite projective dimension as $R$-module. In this manuscript,  we seek an answer to the following question.

\begin{question} \label{question_finite_proj_dim}
Let $M$ be an $R$-module of finite length. What vanishing conditions on $\beta^F_i(M,R)$ imply that $M$ has finite projective dimension?
\end{question}

Miller \cite{ClaudiaMillerMathZ} showed if $R$ is a complete intersection and $M$ is an $R$-module of finite length, then the vanishing of $\beta^F_i(M,R)$ for some $i\gs 1$ implies that $M$ has finite projective dimension. We refer to  \cite{DaoSmirnov} for related results for Gorenstein rings. In Section \ref{SecProjDim}, we answer this question for rings that have small regular algebras, and for rings that have $F$-contributors.
Later, we focus on one-dimensional rings and give the following characterization for the vanishing of $\beta^F_i(M,R).$

\begin{theorem*}(see Theorem \ref{equiv})
Let $(R,\m,K)$ be a one-dimensional local ring of positive characteristic $p$, and let $M$ be an $R$-module of finite length. Let $(G_j,\varphi_j)_{j \gs 0}$ be a minimal free resolution of $M$. Then the following are equivalent:
\begin{enumerate}[(i)]
\item  $\IM(\varphi_{i+1}) \subseteq H^0_\m(G_i)$.
\item  $\Tor_i^R(M,\eee {(R/\p)}) =0$ for all $e \gs 0$, for all $\p \in \MIN(R)$.
\item $\Tor_i^R(M,\eee {(R/\p)}) =0$ for all $e \gg 0$, for all $\p \in \MIN(R)$.
\item $\beta_i^F(M,R) = 0$.
\end{enumerate}
Assume in addition that $R$ is complete and $K$ is algebraically closed. If $V$ denotes the integral closure of $R$ in its ring of fractions, then the conditions above are equivalent to
\begin{enumerate}[(i)]
\setcounter{enumi}{4}
\item  $\Tor_i^R(M,V) = 0$.
\end{enumerate}
\end{theorem*}

As a consequence of this theorem, we show that if $R$ is a one-dimensional Cohen-Macaulay local ring and $\lambda(M)<\infty$, then $\beta^F_i(M,R) = 0$ for any $i \gs 1$ implies that $M$ has finite projective dimension (see Corollary \ref{CorProjCMDimOne}). Furthermore, we prove that the vanishing of two consecutive $\beta^F_i(M,R)$ implies that  $M$ has finite projective dimension in every one-dimensional local ring (see Corollary \ref{CorPrjDimGralOneDim}).

From the theorem above we have that $\beta^F_i(M,R)=0$  if and only the $(i+1)$-syzygy has finite length. On the other hand, there are modules of infinite projective dimension over one-dimensional rings which have second syzygies of finite length (see Example \ref{finitelength}). Motivated by Iyengar's question about the eventual stability of dimensions of syzygies and by our results regarding $\beta^F_i(M,R)$, 
we ask the following question.

\begin{question}\label{Q SyzFinLen}
Let $R$ be a $d$-dimensional local ring, and let $M$ be a finitely generated $R$-module such that $\pd_R(M)=\infty$ and $\lambda(M)<\infty$. If $i>d+1,$ then must the length of the $i$-th syzygy be infinite?
\end{question}

In Section \ref{SecSyzFinLen}, we study this question, mainly for one-dimensional rings. In particular, we show that the answer to Question \ref{Q SyzFinLen} is positive for one-dimensional Buchsbaum rings (see Proposition \ref{Prop Buchsbaum}). We also obtain a partial answer for modules whose Betti numbers are eventually non-decreasing (see Proposition \ref{Prop Betti non decreasing}). Furthermore, we show that the first and third syzygies of $M$ are either zero or have infinite length for every finite length module $M$ over a one-dimensional ring (see Corollary \ref{Cor 1-3}). The assumption of $M$ having finite length is necessary, as shown in Example \ref{Ex 1-3 finite length}. Aside from the study of projective dimension, we study basic properties of the higher invariants that resemble the Hilbert-Kunz multiplicity in other aspects.

\section{Notation and terminology}

Throughout this article, $(R,\m,K)$ will denote a local ring of Krull dimension $\dim(R)= d$. For a finitely generated $R$-module $M$, we define $\dim(M) = \dim(R/(0:_R M))$, where $0:_R~M = \{x \in R \mid xM = 0\}$. When $M = 0$, we set $\dim(M) = -1$. An element $x \in R$ such that $\dim(R/(x)) = d-1$ will be called a {\it parameter of $R$}. Given a finitely generated $R$-module $M$, a {\it minimal free resolution} $(G_\bullet,\varphi_\bullet)$ of $M$ is an exact sequence
\[
\xymatrixcolsep{5mm}
\xymatrixrowsep{2mm}
\xymatrix{
\ldots  \ar[rr] && G_{i+1} \ar[rr]^-{\varphi_{i+1}} && G_i \ar[rr]^-{\varphi_i} && \ldots \ldots \ar[rr] && G_1 \ar[rr]^-{\varphi_1} && G_0 \ar[rr] && M \ar[rr] && 0
}
\]
such that $G_i \cong R^{\beta_i(M)}$ are free $R$-modules and $\IM(\varphi_{i+1}) \subseteq \m G_i$.  The integers $\beta_i(M) = \rk(G_i) = \lambda(\Tor^R_j(M,K))$ are called the {\it Betti numbers of $M$}. If $\beta_i(M)=0$ for some $i$, we say that $M$ has {\it finite projective dimension}, and that it is equal to  $\pd_R(M)=\max\{i\in\NN \mid \beta_i(M)\neq 0\}$. We adopt the convention that $\pd_R(M)=-\infty$, when $M=0$. For all $i \gs 0$ we set $\Omega_i(M) = \CoKer(\varphi_i)$, and we call it the {\it $i$-th syzygy of the module $M$}. Note that $\Omega_0(M) = M$. When no confusion may arise, we will denote $\Omega_i(M)$ simply by $\Omega_i$.

Throughout the manuscript we often make use of local cohomology tools. For every $k\in \NN,$ the quotient map $R/\m^{k+1}\to R/\m^k$ induces maps of functors 
$$\Ext^i_R(R/\m^{k},-)\to \Ext^i_R(R/\m^{k+1},-).$$ For an $R$-module $M$, we define the {\it $i$-th local cohomology of $M$ with support on $\m$} by
$$
H^i_\m(M)=\lim\limits_{k\to \infty}\Ext^i_R(R/\m^k,M).
$$
In particular, $H^0_\m(M)=\bigcup_{k\in\NN} 0:_M \m^k=\{v\in M \mid \m^kv=0\hbox{ for some }k\in\NN\}$. 
For a non-zero finitely generated $R$-module $M$, $\Depth(M)$ denotes the smallest integer $j$ such that $H^j_\m(R) \ne 0$. When $\Depth(M) = \dim(M)$, the module is called {\it Cohen-Macaulay}, and $M$ is called {\it maximal Cohen-Macaulay} if $\Depth(M) = \dim(R)$.

We now review some basic facts about integral closures. For an ideal $I \subseteq R$ and an element $x \in R$ we say that $x$ is integral over $I$ if it satisfies an equation of the form $x^n+r_1x^{n-1}+\ldots + r_n = 0$, where $r_j \in I^j$ for all $j=1,\ldots,n$. The set of elements integral over $I$ forms an ideal, which is called the {\it integral closure of $I$}, and denoted $\overline{I}$. For an ideal $J \subseteq I$, we say that $J$ is a {\it reduction of $I$} if $\overline{J} = \overline{I}$. We say that $J$ is a {\it minimal reduction of $I$} if it is a reduction of $I$ which minimal with respect to containment. We refer the reader to \cite[Chapter 8]{SwansonHuneke} for more details about reductions. For a domain $R$, let $V$ be the integral closure of $R$ in its field of fractions $L$. We define the {\it conductor of $R$} as the set of all elements $z \in L$ such that $z V \subseteq R$, and we denote it by $C$. When $V$ is finite over $R$, one can show that $C$ is the largest ideal which is common to $R$ and $V$, and that $C$ contains a non-zero divisor for $R$ \cite[Exercise 2.11]{SwansonHuneke}. In particular, if $(R,\m,K)$ is an excellent one dimensional local domain, the conductor is $\m$-primary. See \cite[Chapter 12]{SwansonHuneke} for more results about conductors.

We also need the notion of dualizing complex. We refer to \cite[P. 51]{RobertsHom} or to \cite[Chapter V]{Residues} for more details.
\begin{definition}  Let $(S,\n,L)$ be a local ring of dimension $d$. We say that a complex $D^\bullet$ is a dualizing complex of $S$ if
\begin{enumerate}
\item $D^i=\bigoplus_{\dim S/\p=d-i} E_S(S/\p).$
\item The cohomology $H^i(D^\bullet)$ is finitely generated.
\end{enumerate}
\end{definition}
\begin{remark}\label{RemDualComplex}
If $(S,\n,L)$ is a complete ring, then $S$ has a dualizing complex, $D_S^\bullet$ \cite[P. $299$]{Residues}. If $\p$ is a prime ideal such that $\dim S/\p=\dim S$, we have that $S_\p$ is Artinian, hence complete. In addition, $D_{S_\p}^\bullet:=D_{S}^\bullet \otimes S_\p$ is a dualizing complex for $S_\p.$ Furthermore, $H^j(D^\bullet_{S_\p})=H^j(D_S^\bullet )\otimes S_\p=0$ for $j>0$
and $\omega_{S_\p}\cong H^0(D_{S_\p}^\bullet)=E_{S_\p}(S_\p/\p S_\p)$,  because $S_\p$ is Artinian, and so, Cohen-Macaulay.
\end{remark}

We now introduce Buchsbaum rings, which we study 
Question \ref{Q SyzFinLen} in Section \ref{SecSyzFinLen}.
\begin{definition} Let $(R,\m,K)$ be a local ring of dimension $d$. We say that $R$ is a {\it Buchsbaum ring} if, for any system of parameters $x_1,\ldots,x_d$, we have
\[
\ds (x_1,\ldots,x_{i-1}):x_i = (x_1,\ldots,x_{i-1}):\m
\]
for every $i=1,\ldots,d$. When $i=1$, the ideal $(x_1,\ldots,x_{i-1})$ is simply the zero ideal.
\end{definition}
There are several equivalent ways to define Buchsbaum rings, but the one above is the most convenient for our purposes. 
\begin{remark} \label{RemParameter} Let $(R,\m,K)$ be a one-dimensional local ring. Suppose that $R$ is not Cohen-Macaulay, so that $H^0_\m(R) \ne 0$. Then there exists a parameter $x$ of $R$ such that $H^0_\m(R) = 0:_R x$. In fact, fix an integer $n \in \NN$ such that $\m^n H^0_\m(R)$, using that $H^0_\m(R) \subseteq R$ is an ideal, hence it is finitely generated. Take any parameter $y \in \m$, and set $x = y^n$. With this choice, one has $xH^0_\m(R) \subseteq \m^n H^0_\m(R) = 0$, so that $H^0_\m(R) \subseteq 0:_R x$. On the other hand, there exists $k \in \NN$ such that $\m^k \subseteq (x)$. Therefore, if $r \in 0:_R x$, we get $r \m^k \subseteq r(x) = 0$, so that $r \in H^0_\m(R)$. 
We conclude that $H^0_\m(R)= 0:_R x.$
\end{remark} 
\begin{remark} \label{Rem Buchsbaum} Let $(R,\m,K)$ be a one-dimensional Buchsbaum ring. By Remark \ref{RemParameter}, there exists a parameter $x \in R$ such that $0:_R x = H^0_\m(R)$. By definition of Buchsbaum ring, we have that 
\[
\ds H^0_\m(R) = 0:_R x = 0:_R \m.
\]
In particular, $\m H^0_\m(R) = 0$, that is, $H^0_\m(R) \cong \bigoplus_{j=1}^t K$ is a finite dimensional $K$-vector space.
\end{remark}

For the rest of the section, we assume that $(R,\m,K)$ is a local ring of characteristic $p>0$. For an integer $e \gs 1$, we consider the $e$-th iteration of the Frobenius endomorphism $F^e:R \to R$, $F^e(r) = r^{p^e}$ for all $r \in R$. For an $R$-module $M$, one can consider $M$ with the action induced by restriction of scalars, via $F^e$. We denote this module by $\eee M$. More explicitly, for $r \in R$ and $m \in \eee M$, we have $r * m = r^{p^e}m$. 
\begin{definition} We say that $R$ is {\it $F$-finite} if $^1 R$ is a finitely generated $R$-module. 
\end{definition} 
Note that $R$ is $F$-finite if and only if $\eee R$ is a finitely generated $R$-module for any $\gs 1$ or, equivalently, for all $e \gs 1$. Furthermore, $F$-finite rings are excellent \cite[Theorem 2.5]{Kunz1}. When $R$ is $F$-finite, we have that $[K:K^p]< \infty$. In this case, we set $\alpha = \log_p[K:K^p]$.

\section{Definition and properties of $\beta^F_i(M,N)$ and $\mu_i^F(M,N)$}
 We start by defining the Frobenius Betti numbers and showing basic properties that resemble the Hilbert-Kunz multiplicity.
\begin{definition}[see also \cite{Li-PAMS}]
Let $(R,\m,K)$ be a local ring of characteristic $p>0$, let $M$ be an $R$-module of finite length and let $N$ be a finitely generated $R$-module. Define
$$
\beta^F_{i,R}(M,N)=\lim\limits_{e\to \infty} \frac{\lambda(\Tor^R_{i} (M, {^e}N))}{q^{(d+\alpha)}}.
$$ 
We denote $\beta^F_{i,R}(K,R)$ by $\beta^F_{i,R}(R)$ and call it the {\it $i$-th Frobenius Betti number of $R$}. 
If the ring is clear from the context, we only write $\beta^F_{i}(M,N).$
The above limit exists by the main result in \cite{Seibert}. 
\end{definition}
We point out that Li \cite{Li-PAMS} focused on  $\beta_i^F(R/I,R),$  which he denoted by $t_i(I,R).$

\begin{example}
Suppose that $R=S/fS,$ where $S$ is an $F$-finite regular local ring of characteristic $p>0$,  and $f\in S$. We can write $\eee R\cong  R^{a_e}\oplus M_e$, where $M_e$ has no free summands. The limit $s(R):=\lim\limits_{e\to\infty}\frac{a_e}{q^{(d+\alpha)}}$ exists \cite[Theorem 4.9]{KevinFsig}, and it is called the $F$-signature of $R$, which is an important invariant related to strong $F$-regularity \cite[Theorem 0.2]{AberbachLeuschke}.
We consider the minimal free resolution of $\eee R$:
\[
\xymatrixcolsep{5mm}
\xymatrixrowsep{2mm}
\xymatrix{
\ldots  \ar[rr]
&&  R^{\beta_i(\eee R)} \ar[rr] 
&&  R^{\beta_{i-1}(\eee R)} \ar[rr]
&& \ldots \ar[rr]
&& R^{\beta_0(\eee R)} \ar[rr]  
&& \eee R \ar[r] & 0.
}
\]
We note that $\beta_0 (\eee R)=a_e+\beta_0(M_e)$ and $\beta_i(\eee R)=\beta_i(M_e)$ for $i>0.$ Since $M_e$ is a maximal Cohen-Macaulay module with no free summands, we have that $\beta_i(M_e)=\beta_0(M_e)$ for $i>0$ \cite[Proposition 5.3 and Theorem 6.1]{EisenbudBettiHyp}.
Then,
$$
\beta^F_0(R)=e_{HK}(\m,R)=
\lim\limits_{e\to\infty}\frac{\beta_0(\eee R)}{q^{(d+\alpha)}}
=\lim\limits_{e\to\infty}\frac{a_e}{q^{(d+\alpha)}}+\lim\limits_{e\to\infty}\frac{\beta_0(M_e)}{q^{(d+\alpha)}}
=s(R)+\lim\limits_{e\to\infty}\frac{\beta_0(M_e)}{q^{(d+\alpha)}}.
$$
Hence,
$$
\beta^F_i(R)=
\lim\limits_{e\to\infty}\frac{\beta_i(\eee R)}{q^{(d+\alpha)}}
=\lim\limits_{e\to\infty}\frac{\beta_i(M_e)}{q^{(d+\alpha)}}
=\lim\limits_{e\to\infty}\frac{\beta_0(M_e)}{q^{(d+\alpha)}}
=e_{HK}(\m,R)-s(R)
$$
for $i>0.$
\end{example}

As for the Hilbert-Kunz multiplicity, the Frobenius Betti numbers also increase after taking the quotient by a nonzero divisor.

\begin{proposition}
Let $(R,\m,K)$ be a local ring of characteristic $p>0,$
 $M$ be an $R$-module of finite length, and  $x \in \ann(M)$ be a nonzero divisor on $R$. Then
$$\beta^F_{i,R}(M,R)=\lim\limits_{e\to \infty} 
\frac{\lambda(\Tor^R_{i} (M, {^e}R))}{q^{(d+\alpha)}}\ls \beta^F_{i,R/(x)}(N,R/(x))=\lim\limits_{e\to \infty} 
\frac{\lambda(\Tor^{R/(x)}_{i} (M, {^e}(R/(x))))}{q^{(d-1+\alpha)}},$$
 where the subscripts indicate over which ring we are computing the Frobenius Betti  numbers. In particular, $\beta^F_{i,R}(R)\ls \beta^F_{i,R/(x)}(R/(x)).$
\end{proposition}
\begin{proof}
Let $G_\bullet \to {^e}R$ be a minimal free resolution of ${^e}R.$
Let $\overline{R}$ denote $R/xR.$
We have that $\overline{G}_\bullet=G_\bullet \otimes_R \overline{R}$ is a free resolution for ${^e}R\otimes_R \overline{R}$
as an $\overline{R}$-module. Furthermore, we have that $H_0(\overline{G}_\bullet)={^e}R\otimes_R \overline{R}.$ This is a consequence of the fact that $H_i(\overline{G}_\bullet)= \Tor^R_i({^e}R, \overline{R})=0$ for $i>0$ because $x$ is a nonzero divisor on $R$ and ${^e}R$.

Since $x \in \ann(M)$, we have that
$$\Tor^R_i(M,{^e}R)=H_i(M \otimes_R G_\bullet)=H_i(M \otimes_{\ov R} \ov R \otimes_R G_\bullet)=
H_i(M \otimes_{\ov R} \overline{G}_\bullet)
=\Tor^{\overline{R}}_i(M, {^e}R\otimes_R \overline{R}).$$

Since $x$ is a nonzero divisor on $R,$ there is a filtration 
$$
0=L_0\subseteq L_1\subseteq \ldots \subseteq L_{q^{\alpha}}={^e}R\otimes_R \overline{R}
$$
such that $L_{r+1}/L_r={^e}(\overline{R}).$ As a consequence, $\lambda(\Tor^{\overline{R}}_i(M, {^e}R\otimes_R \overline{R}))\ls q^{\alpha}\lambda(\Tor^{\overline{R}}_i(M, {^e}\overline{R})) $

Then 
$$
\lim\limits_{e\to \infty} 
\frac{\lambda(\Tor^R_{i} (M, ^e R))}{q^{(d+\alpha)}}\ls \lim\limits_{e\to \infty} 
\frac{\lambda(\Tor^{\overline{R}}_{i} (M, ^e \overline{R}))}{q^{(d-1+\alpha)}}
$$
\end{proof}

We now introduce  $\mu^F_i(M,N)$, a dual version of $\beta^F_i(M,N)$, which is defined in terms of $\Ext$. In Proposition \ref{HHK Tor=Ext}, we establish a relation between these asymptotic invariants. 

\begin{definition}
Let $(R,\m,K)$ be a local ring of characteristic $p>0,$
let $M$ be an $R$-module of finite length, and let $N$ be a finitely generated $R$-module. We define 
$$
\mu^F_i(M,N)=\lim\limits_{e\to \infty} \frac{\lambda(\Ext^{i}_R (M, {^e}N))}{q^{(d+\alpha)}},
$$ 
\end{definition}

Next, we prove that the numbers $\mu_i^F(M,N)$ are well defined. The proof is essentially the same as the one for $\beta_i^F(M,N)$, as it uses the main result in \cite{Seibert}. Nonetheless, we include it here for completeness.

\begin{proposition}\label{Existence Left}
Let $(R,\m,K)$ be a local ring of characteristic $p>0,$  let $M$ be an $R$-module of finite length, and let $N$ be a finely generated $R$-module.
Then,
$\lim\limits_{e\to \infty} \frac{\lambda (\Ext^i_R(N, ^e M))}{q^{(d+\alpha)}}$ exists. 
Moreover, if $0\to N_1\to N_2\to N_3\to 0$ is a short exact sequence,
then 
$$
\lim\limits_{e\to \infty} \frac{\lambda ( \Ext^i_R(M, {^e}N_2))}{q^{(d+\alpha)}}
=\lim\limits_{e\to \infty} \frac{\lambda (\Ext^i_R(M, {^e}N_1))}{q^{(d+\alpha)}}
+\lim\limits_{e\to \infty} \frac{\lambda (\Ext^i_R(M, {^e}N_3))}{q^{(d+\alpha)}}.
$$
\end{proposition}
\begin{proof} Let $G_\bullet \to M$ be a minimal free resolution of $M$ and define
\[
\ds  g_e(N)=\lambda(H^i(\Hom_R(G_\bullet),\eee N)).
\]
Let  $0\to N_1\to N_2\to N_3\to 0$ be a short exact sequence of finitely generated $R$-modules.
We have that $g_e(N_2)\ls g_e(N_1)+g_e(N_3)$ and equality holds if the sequence splits.
Then, 
$$
\lim\limits_{e\to\infty}\frac{g_e(N)}{q^{(d+\alpha)}}=\lim\limits_{e\to \infty} \frac{\lambda (\Ext^i_R(M, \eee N))}{q^{(d+\alpha)}}
$$
exists, and it is additive in short exact sequences \cite{Seibert}. 
\end{proof}

\begin{proposition} \label{PropAdditive}
Let $(R,\m,K)$ be a local ring of characteristic $p>0,$  $M$ be an $R$-module of finite length, and $N$ be a finitely generated $R$-module. Let $\Lambda$ be the set of all prime ideals $\p$ such that $\dim R/\p=\dim R$. We have that
\[
\ds \beta^F_i(M,N)=\sum_{\p \in \Lambda} \beta^F_i(M,R/\p)\lambda_{R_\p}(N_\p)
\]
and
\[
\ds \mu^F_i(M,N)= \sum_{\p\in\Lambda} \mu^F_i(M,R/\p)\lambda_{R_\p} (N_\p).
\]
\end{proposition}
\begin{proof}
We only prove the first statement, since the proof of the second is completely analogous. Let $0=N_0\subseteq N_1\subseteq\ldots\subseteq N_h=N$ be a  filtration for $N$ such that $N_j/N_{j-1}\cong R/\p_j,$ where $\p_j\subseteq R$ is a prime ideal. Thus, we have short exact sequences $0 \to N_{j-1} \to N_j \to R/\p_j \to 0.$ We deduce that $\beta_i^F(M,N) = \sum_{j=1}^h \beta_i^F(M,R/\p_j)$
 \cite[Proposition 1(b)]{Seibert}. In addition, we have that $\beta_i^F(M,R/\p_j) = 0$ whenever $\dim(R/\p_j) < \dim(R)$ \cite[Proposition 1(a)]{Seibert}. To prove the result, we need to count the number of times that a prime $\p$ such that $\dim R/\p=\dim R$ appears in the prime filtration. This number is obtained by localizing the above filtration at $\p$, and counting the length of the resulting chain. 
Since the localized chain is a composition series of the module $N_\p$, we obtain that the number of times $\p$ appears in the prime filtration above is given by $\lambda_{R_\p} (N_\p)$. Then
$$
\beta^F_i(M,N)=\sum_{{\tiny \begin{array}{c}  j=1 \\ \p_j \in \Lambda \end{array}}}^h \hspace{-0.2cm} \beta_i^F(R/\p_j) = \sum_{\p\in\Lambda} \beta^F_i(M,R/\p)\lambda_{R_\p} (N_\p).
$$ 
\end{proof}

\begin{remark} \label{RemAdditive}
It follows from Proposition \ref{PropAdditive} that, if $\beta^F_i(M,R)=0$ for some $i\in \NN$, we have that $\beta^F_i(N,R/\p)=0$ for every minimal prime of $R$ such that $\dim(R/\p)=d$. Therefore, if this is the case, $\beta^F_i(M,N)=0$ for every finitely generated $R$-module $N$, using again Proposition \ref{PropAdditive}. A similar statement holds for $\mu^F_i(M,R)$.
\end{remark}

The following theorem is related to results of Chang \cite[Lemma 1.20 and Corollary 2.4]{Chang}, and in some cases it follows from them. We present a different proof that does not use spectral sequences.
\begin{theorem}\label{Thm Bass}
Let $(R,\m,K)$ be a local ring of characteristic $p>0,$
 $M$ be an $R$-module of finite length, and  $N$ be a finitely generated $R$-module.
Then
$$
\lim\limits_{e\to\infty}\frac{\lambda(\Ext^i_R(M,{^e}N))}{q^{(i+1+\alpha)}}=0
$$
for $i<d.$ In particular, 
$\mu^F_i(M,N)=0$ for $i<d.$
\end{theorem}
\begin{proof}
Our proof will be by induction on $n=\dim(N).$

If $n=0,$ we have that $h=\lambda(N)$ is finite.
There is a filtration
$$
0=N_0\subseteq N_1\subseteq \ldots \subseteq N_h=N
$$
such that $N_{j}/N_{j-1}\cong K$. From the short exact sequences $0\to N_{j-1}\to N_{j}\to K\to 0$, we have that
$$
\lambda(\Ext^i_R(M,{^e}N_j)))\ls \lambda(\Ext^i_R(K,{^e}N_{j-1}))+\lambda(\Ext^i_R(K,{^e}K)).
$$
Since
$$
\lim\limits_{e\to \infty} \frac{\lambda(\Ext^i_R(M,{^e}K))}{q^{(i+1+\alpha)}} =\lim\limits_{e\to\infty} \frac{q^{\alpha}\lambda(\Ext^i_R(M,K))}{q^{(i+1+\alpha)}} =\lim\limits_{e\to\infty} \frac{\lambda(\Ext^i_R(M,K))}{q^{(i+1)}} =0,
$$
we have that
$$
\lim\limits_{e\to \infty} \frac{\lambda(\Ext^i_R(M,{^e}N))}{q^{(i+1+\alpha)}}=0
$$
by an inductive argument.

Suppose that our claim is true for modules of dimension less or equal to $n-1.$
There is a filtration
$$
0=N_0\subseteq N_1\subseteq\ldots \subseteq N_h=N
$$
such that $N_{j}/N_{j-1}\cong R/\p_{j}$, where $\p_{j}\subset R$ is a prime ideal. From the short exact sequences $0\to N_{j-1}\to N_{j}\to R/\p_j\to 0$ we have that
$$
\lambda(\Ext^i_R(M,{^e}N_{j}))\ls \lambda(\Ext^i_R(M,{^e}N_{j-1}))+\lambda(\Ext^i_R(M,{^e}(R/\p_j))).
$$
It suffices to show that 
\begin{equation}\label{WTS Ind i}
\lim\limits_{e\to \infty} \frac{\lambda(\Ext^i_R(M,{^e}(R/\p_j)))}{q^{(i+1+\alpha)}}=0
\end{equation}
for primes $\p_j$ such that $\dim_R(R/\p_j)=n=\dim_R N.$
Let $T=R/\p_j.$ Let $x\in \Ann_R M \setminus \p_j$, which exists because $\dim_R T=\dim_R N>0=\dim_R M.$ 
We have a short exact sequence
\[
\xymatrixcolsep{5mm}
\xymatrixrowsep{2mm}
\xymatrix{
0 \ar[rr] && {^e}T  \ar[rr]^-{x} && {^e}T \ar[rr] && {^e}T/x({^e}T) \ar[rr] && 0,
}
\]
which induces a long exact sequence
\begin{equation}\label{Ind Ext}
\xymatrixcolsep{5mm}
\xymatrixrowsep{2mm}
\xymatrix{
\ldots \ar[r]& \Ext^i_R(M,{^e}T) \ar[rr]^-{0} && \Ext^i_R(M,{^e}T) \ar[rr] &&  \Ext^i_R(M,{^e}T/x({^e}T))\ar[r] & \ldots.
}
\end{equation}
Then, for every $i$,
$$
\lambda(\Ext^i_R(M,{^e}T))\ls\lambda(\Ext^{i-1}_R(M,{^e}T/x({^e}T))).
$$
We have a filtration 
$$
0=L_0\subseteq L_1\subseteq \ldots \subseteq L_{q^{\alpha}}={^e}T/x({^e}T)
$$
such that $L_{r+1}/L_r={^e}(T/xT)$ because $x$ is not a zero divisor of $T.$
From the induced long exact sequence by $\Ext^i_R(M,-),$
we have that
$$\lambda(\Ext^i_R(M,{^e}T/x({^e}T)))\ls q^\alpha
\lambda(\Ext^i_R(M,{^e}(T/xT)))).$$ 
Therefore
\begin{align*}
\lim\limits_{e\to\infty}\frac{\lambda(\Ext^{i}_R(M,{^e}T))}{q^{(i+1+\alpha)}}
&\ls\lim\limits_{e\to\infty}\frac{\lambda(\Ext^{i-1}_R(M,{^e}T/x({^e}T)))}{q^{(i+1+\alpha)}}\\
&\ls\lim\limits_{e\to\infty}\frac{q^{\alpha}\lambda(\Ext^{i-1}_R(M,{^e}(T/xT)))}{q^{(i+1+\alpha)}}\\
&=\lim\limits_{e\to\infty}\frac{\lambda(\Ext^{i-1}_R(K,{^e}(T/xT))))}{q^{(i+\alpha)}}\\
&=0\hbox{ because }\dim T/xT=n-1.
\end{align*}
\end{proof}

\begin{corollary} \label{F-contributor}
Let $(R,\m,K)$ be a local ring of characteristic $p>0$. Let $N$ be a finitely generated $R$-module, and let $C$ be an $R$-module such that, for all $e \gg 0$, $C^{\theta_e}$ is a direct summand of ${^e}N$ for some $\theta_e \in \NN$. Assume that $\theta=\limsup_{e\to\infty}\frac{\theta_e}{q^{(d+\alpha)}} > 0$. Then, for all $R$-modules $M$ of finite length, and all integers $i$, we have
$$\mu^F_i(M,N)\gs\theta \cdot \lambda(\Ext^i_R(M,C)).$$
In particular, $C$ is a maximal Cohen-Macaulay module.
\end{corollary}
\begin{proof}
We have 
\[
\ds  \mu^F_{i}(M,N) = \lim_{e \to \infty} \frac{\lambda(\Ext^{i}_R(M,{^e}N))}{q^{(d+\alpha)}} \gs\limsup_{e \to \infty} \frac{\theta_e \cdot \lambda(\Ext^{i}_R(M,C))}{q^{(d+\alpha)}} = \theta \cdot \lambda(\Ext^{i}_R(M,C)).
\]

Using $M=K$ in Theorem \ref{Thm Bass}, we obtain that $\mu^F_{i}(K,N)=0$ for all $i<d$. It follows from the inequality that $\Ext^i_R(K,C)=0$ for all $i<d$, and then $C$ is a maximal Cohen-Macaulay module.
\end{proof}

\begin{remark}\label{Yao}
Let $(R,\m,K)$ be a local ring of characteristic $p>0,$
and $N$ be a finitely generated $R$-module. We say that an $R$-module $C$ is an $F$-contributor of $N$ if $C^{\theta_e}$ is a direct summand of $\eee N$  for $e \gg 0$, and $\limsup_{e\to\infty}\frac{\theta_e}{q^{(d+\alpha)}} > 0$ \cite{YaoFcontributors}. Corollary \ref{F-contributor} shows that every $F$-contributor is a maximal Cohen-Macaulay module. This was already noted by Yao \cite[Lemma 2.2]{YaoFcontributors} when $N$ has finite $F$-representation type. 
\end{remark}

The following proposition shows that taking limits with respect to $\Tor$ or $\Ext$ give the same invariants up to a shift in the homological degrees.

\begin{proposition}\label{HHK Tor=Ext}
Let $(R,\m,K)$ be a local ring of characteristic $p>0,$
and $M$ be an $R$-module of finite length.
Then
$$
\beta^F_i(M,R)=\mu^F_{d+i}(M,R)
$$
for every $i\in\NN.$
\end{proposition}
\begin{proof}
Since $\beta^F_i(M,R)$ and $\mu^F_{d+i}(M,R)$ are not affected by completion at $\m$, we may assume that $R$ is a complete local ring. In this case, $R$ has a dualizing complex $D_R^\bullet$ by Remark \ref{RemDualComplex}. We have that 
$$
\beta^F_i(M,R)=\mu^{F}_{d+i}(M,H^0(D_R^\bullet))
$$
by \cite[Proposition 2.3(2)]{Chang}. 
Let $\Lambda$ be the set of all prime ideals of $R$ such that $\dim R/\p=\dim R.$ Let $\p\in\Lambda.$
We have that $(H^0(D_R^\bullet))_\p=H^0(D_{R_\p}^\bullet)=\omega_{R_\p}$ by Remark \ref{RemDualComplex}.
We have that
$\omega_{R_\p}=\Hom_{R_\p}(R_\p,E_{R_\p}(R_\p/\p R_\p))$ and $\lambda_{R_\p}(\omega_{R_\p})=\lambda_{R_\p}(R_\p).$ Finally, by Proposition \ref{PropAdditive}
\begin{align*}
\mu^F_{d+i}(M,H^0(D_R^\bullet)) &=\sum_{\p\in\Lambda} \mu^F_{d+i}(M,R/\p)\lambda_{R_\p}(H^0(D_{R_\p}^\bullet))\\
&=\sum_{\p\in\Lambda} \mu^F_{d+i}(M,R/\p)\lambda_{R_\p}(\omega_{R_\p})\\
&=\sum_{\p\in\Lambda} \mu^F_{d+i}(M,R/\p)\lambda_{R_\p}(R_\p)\\
&=\mu^F_{d+i}(M,R).
\end{align*}
\end{proof}

\begin{remark} If $R$ itself has an $F$-contributor $C$, then we get a relation involving the $\beta_i^F$'s. In fact, by Proposition \ref{HHK Tor=Ext}, we have $\beta_i^F(M,R) = \mu_{d+i}^F(M,R)$ for all $i \in \NN$. Thus, in the notation of Corollary \ref{F-contributor}, we have $\beta_i^F(M,R) \gs\theta \cdot \lambda(\Ext^{d+i}_R(M,C))$.
\end{remark}
We end this section with a proposition that shows how $\beta^F_i(M,N)$ behaves under some flat ring extensions. First, we need a different way to compute $\beta^F_i(M,N)$.
\begin{remark}\label{OtherDef}
Let $(R,\m,K)$ be a local ring of characteristic $p>0$, let $M$ be an $R$-module of finite length, and let $N$ be a finitely generated $R$-module. Let $G_\bullet=(G_j,\varphi_j)_{j \gs 0}$ denote a minimal free resolution of $M$. Let $G^{[q]}_\bullet$ be the complex $(G_j,\varphi^{[q]}_j)_{j \gs 0}$, where the matrix of $\varphi^{[q]}_j$ has as entries the $q$-th powers of the entries in the matrix of $\varphi_j$. 
We have that
$$
\lambda(\Tor^R_i(M, {^e}N))=q^\alpha \lambda(H_i(G^{[q]}_\bullet\otimes_R N)). 
$$
Hence, 
$$
\beta^F_i(M,N)=
\lim\limits_{q\to \infty} 
\frac{\lambda(H_i(G^{[q]}_\bullet\otimes_R N))}{q^d}.
$$ 

\end{remark}
\begin{proposition}\label{Extension}
Let $(R,\m,K)\to (S,\n,L)$ be a flat extension of two $F$-finite local rings of characteristic $p>0.$ Let $M$ be a finite length $R$-module.
Let $\alpha=\log_p[K:K^p]$ and $\theta=\log_p[L:L^p].$
Suppose that $\m S=\n.$ Then,
$$
\beta^F_{i,R}(M,R)=\lim\limits_{e\to \infty} \frac{\lambda(\Tor^R_{i} (M, {^e}R))}{p^{e(d+\alpha)}}=\lim\limits_{e\to \infty} \frac{\lambda(\Tor^S_{i} (M\otimes_R S, {^e}S))}{p^{e(d+\theta)}}=\beta^F_{i,S}(M\otimes_R S,S).
$$ 
In particular, we have that $\beta_{i,R}^F(M,R) = \beta_{i,\widehat{R}}^F(\widehat{M},\widehat{R})$.
\end{proposition}
\begin{proof}
Let $q=p^e$. We have that
\begin{align*}
\frac{\lambda_R(\Tor^R_i(M, {^e}R))}{q^{\alpha}}
&=\lambda_R(H_i(G^{[q]}_\bullet))\hbox{ by Remark \ref{OtherDef}.}\\
&=\lambda_S(H_i(G^{[q]}_\bullet\otimes_R S ))\hbox{ because }S\hbox{ is flat and }\m S=\n.\\
&=\lambda_S(H_i((G_\bullet\otimes_R S)^{[q]}))\hbox{ because }G_\bullet\hbox{ is free}.\\
&=\frac{\lambda_S(\Tor^S_i(M\otimes_R S, {^e}S)}{q^{\theta}}\hbox{ by Remark \ref{OtherDef} and because $S$ is flat.}
\end{align*}
After dividing by $q^d$ and taking limits, we have that 
$$
\beta^F_{i,R}(M,R)=\beta^F_{i,S}(M\otimes_R S, S).
$$ 
\end{proof}

\section{Relations with projective dimension}\label{SecProjDim}

Let $(R,\m,K)$ be a local $F$-finite ring of characteristic $p>0$,
and let $M$ be an $R$-module of finite length. In this section we investigate when the vanishing of $\beta^F_i(M,R)$ detects whether $M$ has finite projective dimension.

We first recall known results in this direction. We have that $R$ is a regular ring if and only if $\beta_i^F(R) = \beta_i^F(K,R)=0$ for some $i\gs 1$ \cite[Corollary 3.2]{AberbachLi}. Let $M$ be a finitely generated $R$-module. If $M$ has finite projective dimension, then $\Tor_i^R(M,{^e}R)=0$ for all $i> 0$ and all $e \gs 0$ \cite[Th\'eor\`eme 1.7]{P-S}. Conversely, if $\Tor_i^R(M,{^e}R)=0$ for infinitely many $e$ and all $i> 0$, then $M$ has finite projective dimension \cite[Theorem $3.1$]{Herzog}. In fact, even more is true: if $\Tor_i^R(M,\eee R) = 0$ for $\Depth(R)+1$ consecutive values of $i$ and some $e \gg 0$, then $M$ has finite projective dimension \cite[Proposition 2.6]{KohLee} (see also \cite[Theorem 2.2.8]{MillerSurvey}). Now, suppose that $R$ is a complete intersection. If $\beta^F_i(M,R)=0$ for some $i >0$, then $M$ has finite projective dimension by \cite[Corollary 2.5]{ClaudiaMillerMathZ} (see also \cite[Corollary 4.11]{DaoSmirnov}).

\begin{proposition}
Let $(R,\m,K)$ be a local ring of characteristic $p>0$, and let $M$ be an $R$-module of finite length. Suppose that there is a regular local ring $(A,\n,L)$ and a map of local rings $\phi:R\to A$ such that $A$ is finitely generated as an $R$-module, and $\dim A=d$. If $$\beta^F_j(M,R)=\beta^F_{j+1}(M,R)=\ldots=\beta^F_{j+d}(M,R)=0,$$
then $M$ has finite projective dimension.
\end{proposition}
\begin{proof}
We note that $\log_p[L:L^p] = \log_p[K:K^p]=\alpha < \infty$, and so, $A$ is $F$-finite. Since $A$ is regular and local, ${^e}A \cong \bigoplus^{q^{(d+\alpha)}} A$. Let $x_1,\ldots,x_d\in A$ be a set of generators for $\n$, and let $I_r:=(x_1,\ldots,x_r)A$. By induction on $r$ we will show that 
\begin{equation}\label{Eq t reg}
\Tor_{j+r}^R(M,A/I_r)=\ldots=\Tor_{j+d}^R(M,A/I_r)=0
\end{equation}
for every $r$.
If $r=0$, we have that  $\Tor^R_i(M, {^e}A)=\oplus^{q^{(d+\alpha)}} \Tor^R_i(M,A)$ for every $i\in\NN$. Then, $\lambda(\Tor^R_i(M, {^e}A))=q^{(d+\alpha)}\lambda(\Tor^R_i(M,A))$, and thus
\[
\ds \beta^F_i(M,A)=\lambda(\Tor^R_i(M,A)).
\]
Since $A$ is finitely generated, and since $\beta^F_i(M,R)=0$ for $i=j,\ldots,j+d$ by assumption, we have that $\beta^F_{j}(M,A)=\ldots =\beta^F_{j+d}(M,A)=0$ by Remark \ref{RemAdditive}. Hence, $\Tor_{j}^R(M,A)=\ldots=\Tor_{j+d}^R(M,A)=0$.
We suppose that (\ref{Eq t reg}) holds for $r-1$ and prove it for $r$. We have a short exact sequence
\[
\xymatrixcolsep{5mm}
\xymatrixrowsep{2mm}
\xymatrix{
0 \ar[rr] && A/I_{r-1} \ar[rr]^-{x_r} && A/I_{r-1}\ar[rr]&& A/I_{r}\ar[rr] && 0.
}
\]
This induces a long exact sequence
\[
\xymatrixcolsep{5mm}
\xymatrixrowsep{2mm}
\xymatrix{
\ldots \ar[r]& \Tor^R_i(M,A/I_{r-1}) \ar[r]^-{x_r} & \Tor^R_i(M,A/I_{r-1})\ar[r] & \Tor^R_i(M,A/I_{r})\ar[r]& \Tor^R_{i-1}(M,A/I_{r-1})\ar[r]& \ldots.
}
\]
Since $\Tor_{j+r-1}^R(M,A/I_{r-1})=\ldots=\Tor_{j+d}^R(M,A/I_{r-1})=0$, we have that $\Tor_{j+r}^R(M,A/I_r)=\ldots=\Tor_{j+d}^R(M,A/I_r)=0$, proving the claim. In particular, we get $\Tor^R_{j+d}(M,A/I_d)=0$. Since $L=A/I_d$ is a finite field extension of $K$, we have
\[
\ds 0=\lambda(\Tor^R_{j+d}(M,A/I_d))=[L:K]\cdot\lambda(\Tor^R_{j+d}(M,K)).
\]
Therefore, $\Tor^R_{j+d}(M,K)=0$ and $M$ has finite projective dimension.
\end{proof}

\begin{lemma}\label{ProjDim Cont} 
Let $(R,\m,K)$ be a local ring of characteristic $p>0.$ 
Suppose that there is an $R$-module $N$ of dimension $d$ that has an $F$-contributor $C$.
Let $M$ be an $R$-module of finite length. If $\beta^F_i(M,N)=0,$ then $\Tor_i^R(M,{^e}C)=0$ for every $e\gs 0$. In particular, if $R$ is strongly $F$-regular of positive dimension $d$, and $\beta^F_i(M,R)=0$ for $d$ consecutive values of $i$, then $M$ has finite projective dimension.
\end{lemma}
\begin{proof}
 For $e' \gg 0$ and $q' = p^{e'}$,  we have that $C^{\theta_{e'}}$ is a direct summand of ${^{e'}} N$, for some $\theta_{e'} \in \NN$ such that $\limsup \frac{\theta_{e'}}{q'^{(d+\alpha)}}>0$. 
We note that $(\eee C)^{\theta_{e'}}$ is a direct summand of ${^{e+e'}}N$ for all $e \gs 0$. Then
\[
\ds \left(\limsup_{e' \to \infty} \frac{\theta_{e'}}{q'^{(d+\alpha)}} \right) \frac{\lambda(\Tor_i^R(M,\eee C))}{q^{(d+\alpha)}} \ls \lim\limits_{e' \to\infty} \frac{\lambda(\Tor_i^R(M,{^{e+e'}}N))}{qq'^{(d+\alpha)}} = \beta_i^F(M,N) = 0.
\]
It follows that $\Tor_i^R(M,{^{e}}C)=0$. If $R$ is strongly $F$-regular, then $R$ is an $F$-contributor of itself. In addition, $R$ is Cohen-Macaulay, and if $\Tor_i(M,{^e}R) = 0$ for $d$ consecutive values of $i$ and for $e\gg 0$, we have that $M$ has finite projective dimension \cite[Proposition 2.6]{KohLee} (see also \cite[Theorem 2.2.11]{MillerSurvey}).
\end{proof}

\begin{proposition} \label{int_closed} \cite[Corollary 3.3]{CHKV} 
Let $(R,\m,K)$ be a local ring, let $I$ be an integrally closed $\m$-primary ideal, and let $N$ be a finitely generated $R$-module. Then $\pd_R(N) < i$ if and only if $\Tor_i^R(N,R/I) = 0$.
\end{proposition}
In particular, Proposition \ref{int_closed} shows that, if $\Tor_i^R(R/I,{^e}R) = 0$ for some $e \gs 1$, then $R$ is regular \cite[Theorem 2.1]{Kunz}. We now present a similar result for Frobenius Betti numbers.

\begin{proposition}
Let $(R,\m,K)$ be a reduced local ring of characteristic $p>0$. Suppose that there exists an $R$-module $N$ of dimension $d$ that has an $F$-contributor $C$. If $I$ is an integrally closed $\m$-primary ideal such that $\beta^F_i(R/I,N)=0$ for some $i>0$, then $R$ is regular.
\end{proposition}
\begin{proof}
By Lemma \ref{ProjDim Cont}, we have that
$\Tor_i(R/I, {^e}C) = 0$ for every $e\gs 0$, and thus ${^e}C$ has finite projective dimension by \cite[Corollary 3.3]{CHKV}. Since ${^e}C$ is a maximal Cohen-Macaulay module \cite[Lemma 2.2]{YaoFcontributors} (see Remark \ref{Yao}), 
we have that ${^e}C$ is a free module for every $e \gs 0$.
In particular, ${^1}C \cong \bigoplus_\alpha R$ and ${^2}C={^1}\left(\bigoplus_\alpha R\right) \cong \bigoplus_\alpha {^1}R$ is free as well. Therefore, ${^1}R$ is free and $R$ is regular \cite[Theorem $2.1$]{Kunz}.
\end{proof}

We now focus on one-dimensional rings. In this case, we can find a characterization of the vanishing of $\beta^F_i(M,R)$. We first  prove  two lemmas. 

\begin{lemma} \label{qthPowerParameter} 
Let $(R,\m,K)$ be a one-dimensional  complete local domain of characteristic $p>0,$ with $K$ algebraically closed. Then there exists a parameter $x \in R$ such that $(x^q) = \m^{[q]}$ for all $q = p^e \gg 0$. Furthermore, if $V$ denotes the integral closure of $R$ in its field of fractions, then $\eee R \cong \bigoplus V$ for all $e \gg 0$ (as $R$-modules).
\end{lemma} 
\begin{proof}  
Since $R$ is a complete domain, we have that $(V,\m_V,K)$ is a one dimensional, integrally closed, local domain. Hence, $V$ is a DVR. Let $x \in R$ be a minimal reduction of $\m$, and let $v$ denote the order valuation on $V$. Let $x,y_1,\ldots,y_n$ be a minimal generating set of the maximal ideal. 
We claim that we can choose the elements $y_i$'s such that $v(x)<v(y_i)$ for all $i=1,\ldots,n$. We have $v(x) \ls v(y_i)$ for all $i$ because  $x$ is a minimal reduction of $\m$ \cite[Proposition 6.8.1]{SwansonHuneke}. If equality holds, say for $i=1$, we have that $y_1/x = \alpha \in K_V = K$ since $K$ is algebraically closed. Fix a lifting $u\in R$ of $\alpha$. If we replace $y_1$ for $y_1':=y_1-ux$ we have that $x,y_1',\ldots,y_n$ is still a minimal generating set of $\m$. Now $v(x)<v(y_1')$, since $y_1'/x \in \m_V$. 
Similarly, if necessary, we may replace each $y_i$  to obtain our claim. Since the conductor $C$ is $\m_V$-primary, for all $e \gg 0$ and all $i=1,\ldots,n$, we have that $(y_i/x)^q =r_i \in \m_V^{[q]} \subseteq C \subseteq R$. Thus $y_i^q =  r_ix^q \in (x)^q$. This shows the first part of the lemma. 

We now focus on the second part of the lemma. Since $K$ is algebraically closed, $R$ and $K$ have the same residue field. It then follows that $R \subseteq V=R+\m_V$. Since $R$ is a domain, we can identify $\eee R$ with $R^{1/q}$, the ring of $q$-th roots of $R$. 
For $w \in V$, we can write $w = u+v$, for some $u \in R$ and $v\in \m_V$. Therefore,  we have that $\m_V^{[q]} \subseteq C \subseteq R$ for $e \gg 0$, because $C$ is $\m_V$ primary. 
This shows that $w^{q} = u^{q}+v^{q} \in R$, that is $w \in R^{1/q}$. Thus, for $e \gg 0$, we have $R \subseteq V \subseteq \eee R$. Hence, $\eee R$ is a $V$-module. Since $V$ is a DVR, $\eee R$  decomposes into a $V$-free part and a $V$-torsion part. Therefore $\eee R$  is torsion free as a $V$-module because $R$ is a domain. Thus, $\eee R \cong \bigoplus_{q} V$. Finally, the $V$-module structure on $\eee R$ is compatible with the inclusion $R \subseteq V$; therefore, $\eee R \cong \bigoplus V$ is also an isomorphism of $R$-modules.
\end{proof}

\begin{lemma} \label{positive} 
Let $(R,\m,K)$ be a one-dimensional local ring of  characteristic $p>0$. Let $(G_j,\varphi_j)_{j\gs 0}$ be a minimal free resolution of a finite length $R$-module $M$. Suppose there exists $i \gs 0$ such that $\IM(\varphi_{i+1}) \not\subseteq \p G_i$ for some $\p \in \MIN(R)$. Then
\[
\ds \beta_i^F(M,R) =  \lim_{e \to \infty} \frac{\lambda(\Tor_i^R(M,\eee R))}{q^{\alpha}} > 0.
\]
\end{lemma}
\begin{proof} 
By the Cohen Structure Theorem, we have that $\widehat{R}=K\ps{x_1,\ldots,x_n}/I$ for some $n\in\NN$ and some ideal $I\subseteq K\ps{x_1,\ldots,x_n}$. Let $S=L\ps{1,\ldots,x_n}/I',$ where $L$ is the algebraic closure of $K$ and $I'=I ~L\ps{1,\ldots,x_n}$.
Every inclusion $K\to L$ gives a flat extension $R\to S$ such that $\m S$ is the maximal ideal of $S$.
If $\IM(\varphi_{i+1} \otimes_R~1_{S}) = \IM(\varphi_{i+1}) \otimes_R S$ is contained in in a minimal prime of S, then $\IM(\varphi_{i+1})$ must be contained in the contraction of such minimal prime to $R$. Then, we can assume that  $R$ is complete and that $K$ is algebraically closed by Proposition \ref{Extension}. 

Let $\ov R$ denote $R/\p$,   $\ov{x}$  the class of the element $x$ modulo $\p$, and $V$  the integral closure of $\ov R$. Since $R/\p$ is a one-dimensional complete local domain, by Lemma \ref{qthPowerParameter} we can choose $0 \ne \overline{x} \in \ov R$ a minimal reduction of $\ov \m:= \m/\p$ and $q_0 = p^{e_0}$ such that $\ov{\m}^{[q]} = (\ov{x}^{q})$ for $q\gs q_0$. We may also choose $q_0$ large enough so that $\ov x^{q}V\cap \ov R\subseteq \ov x R$, by using the Artin-Rees Lemma and the fact that the conductor from $\ov R$ to $V$ is primary to the maximal ideal. In particular, $(\ov x^{q}V:_V\ov r)\subseteq \m_V$ for every $\ov r\in \ov R$ such that $\ov r\notin \ov x\ov R$, where
$\m_V$ is the maximal ideal of $V$, which is a DVR. 

Fix $q \gs q_0$ and consider the matrix associated to $\ov {\varphi}_{i+1}^{[q]} := \varphi_{i+1}^{[q]} \otimes 1_{\ov R}$. Since $q \gs q_0$,  
$\IM(\varphi_{i+1}^{[q]} \otimes 1_{\ov R}) \subseteq {\ov \m}^{[q]}G_i = (\ov x^{q})G_i$. Because $\IM(\varphi_{i+1}) \not\subseteq \p G_i$, by changing the basis for $G_{i+1}$ if needed, we can assume that  the matrix
\begin{eqnarray*}
\ov \varphi_{i+1}^{[q]} = \ov x^{q+j} \left[
\begin{array}{c|ccc}
\ov{r}_1& * & \ldots & * \\ \hline
\ov{r}_2 & * & \ldots & * \\ \hline \vdots & \vdots&& \vdots  \\   \hline \ov{r}_n & * &\ldots&*
\end{array},
\right]
\end{eqnarray*}
where we have factored out the biggest possible power of $\ov{x}$, so that $\ov{r}_1 \notin (\ov{x})$.  Here $n = \rk(G_i)$.

Let $q' = p^{e'},$ and consider the matrix associated to $\ov \varphi_{i+1}^{[qq']}$:
\begin{eqnarray}
\label{image}
\ov \varphi_{i+1}^{[qq']} = \ov{x}^{(q+j)q'} \left[
\begin{array}{c|ccc}
\ov{r}_1^{q'} & * & \ldots & * \\ \hline
\ov{r}_2^{q'} & * & \ldots & * \\ \hline \vdots & \vdots&& \vdots  \\   \hline \ov{r}_n^{q'} & * &\ldots&*
\end{array}
\right].
\end{eqnarray}
We claim that $[\ov{r}_1^{q'},\ov{r}_2^{q'},\ldots,\ov{r}_n^{q'}]^T \in \Ker\left(\ov {\varphi}_{i}^{[qq']}\right)$. In fact, we have that
\begin{eqnarray*}
\ov{x}^{qq'+jq'}\left[
\begin{array}{c}
{\ov{r}_1}^{q'} \\
{\ov{r}_2}^{q'} \\
\vdots \\
{\ov{r}_n}^{q'} \\
\end{array}
\right] \in \IM\left(\ov {\varphi}_{i+1}^{[qq']}\right) \subseteq \Ker\left(\ov {\varphi}_{i}^{[qq']}\right);
\end{eqnarray*}
therefore,
\begin{eqnarray*}
\ov {\varphi}_{i}^{[qq']} \left(\ov{x}^{qq'+jq'} \cdot \left[
\begin{array}{c}
{\ov{r}_1}^{q'} \\
{\ov{r}_2}^{q'} \\
\vdots \\
{\ov{r}_n}^{q'} \\
\end{array}
\right]
\right) = \ov{x}^{qq'+jq'} \cdot \ov {\varphi}_{i}^{[qq']} \left(\left[
\begin{array}{c}
{\ov{r}_1}^{q'} \\
{\ov{r}_2}^{q'} \\
\vdots \\
{\ov{r}_n}^{q'} \\
\end{array}
\right]
\right) = 0 .
\end{eqnarray*}
Since $\ov{x}^{qq'+jq'}$ is a nonzero divisor in $\ov R$, we  have
\begin{eqnarray*}
\ov {\varphi}_{i}^{[qq']} \left(\left[
\begin{array}{c}
{\ov{r}_1}^{q'} \\
{\ov{r}_2}^{q'} \\
\vdots \\
{\ov{r}_n}^{q'} \\
\end{array}
\right]
\right) =0,
\end{eqnarray*}
which proves the claim. Thus
\begin{eqnarray*}
\begin{array}{rl}
\ds \lambda(\Tor_i^R(M,\eeep{(R/\p)})) \gs & \ds \lambda\left(\frac{\ov R[\ov{r}_1^{q'},\ldots,\ov{r}_n^{q'}]^T + \IM\left(\ov {\varphi}_{i+1}^{[qq']}\right)}{\IM\left(\ov {\varphi}_{i+1}^{[qq']}\right)}\right)  \\ \\
\ds \gs & \ds \lambda \left(\frac{(\ov{r}_1^{q'}) + (\ov{x}^{qq'})}{(\ov{x}^{qq'})}\right),
\end{array}
\end{eqnarray*}
because $\IM\left(\ov {\varphi}_{i+1}^{[qq']}\right) \subseteq (\ov{x}^{qq'})G_i$. This  comes from the expression of $\ov {\varphi}_{i+1}^{[qq']}$ in (\ref{image}). We also have projected onto the first component of $G_i$. We now have a cyclic module which is isomorphic to the quotient of $\ov R$ by the ideal $(\ov{x}^{qq'} : \ov{r}_1^{q'})$.

We claim that there exists an integer $q_1=p^{e_1}$ such that for all $q'$, 
\[
(\ov{x}^{qq'} : \ov{r}_1^{q'}) \subseteq (\ov{x}^{q'/q_1}).
\]
Assuming the claim and  lifting back to $R$, we get

\begin{eqnarray*}
\begin{array}{rl}
\ds \lambda(\Tor_i^R(M,\eeep{\ov R})) \gs &  \ds \lambda \left(\frac{(\ov r_1^{q'}) + (\ov x^{qq'})}{(\ov x^{qq'})}\right) \\ \\
& \ds \gs \lambda \left( \frac{\ov R}{(\ov x^{q'/q_1}) } \right).
\end{array}
\end{eqnarray*}
Dividing by $qq'$ and taking the limit as $e' \to \infty$, we get
\begin{align*}
\ds \beta_i^F(M,\ov R) &= \lim_{e' \to \infty} \frac{\lambda(\Tor_i^R(M,\eeep(\ov R)))}{qq'} \\
& \gs \lim_{e' \to \infty} \frac{\lambda(\ov R/(\ov x^{q'/q_1}))}{qq'} =  \frac{1}{qq_1}\ee_{HK}(x,\ov R) > 0.
\end{align*}

Since  $\dim(R/\q) = \dim(R)$ for $\q \in \Spec(R)$ if and only if $\q \in \MIN(R)$,  we have
\[
\ds \beta_i^F(M,R) = \sum_{\q \in \MIN(R)} \left(\beta_i^F(M,R/\q) \lambda(R_\q) \right)\gs \beta_i^F(M,R/\p) > 0.
\]
by Proposition \ref{PropAdditive}.

It remains to prove the claim.   Suppose that $u\in (\ov{x}^{qq'} : \ov{r}_1^{q'})$. Then $u\in (\ov{x}^{qq'} : \ov{r}_1^{q'})V\cap \ov R =(\ov{x}^{q}V :_V \ov{r}_1)^{[q']}\cap \ov R\subseteq
\m_V^{q'}\cap R $ by the choice of $q$. Since the conductor of $\ov R$ is primary to the maximal ideal it follows that there exists a $q_1= p^{e_1}$ such that
$\m_V^{q'}\cap \ov R \subseteq (\ov x^{q'/q_1})$ as claimed.

\end{proof}

\begin{theorem}\label{equiv} 
Let $(R,\m,K)$ be a one-dimensional local ring of  characteristic $p>0$, and  $M$ be an $R$-module of finite length. Let $(G_j,\varphi_j)_{j \gs 0}$ denote a minimal free resolution of $M$. Then the following are equivalent:
\begin{enumerate}[(i)]
\item \label{thm1} $\IM(\varphi_{i+1}) \subseteq H^0_\m(G_i)$.
\item \label{thm2} $\Tor_i^R(M,\eee {(R/\p)}) =0$ for all $e \gs 0$, for all $\p \in \MIN(R)$.
\item \label{thm3} $\Tor_i^R(M,\eee {(R/\p)}) =0$ for all $e \gg 0$, for all $\p \in \MIN(R)$.
\item \label{thm4} $\beta_i^F(M,R) = 0$.
\end{enumerate}
Assume in addition that $R$ is complete and $K$ is algebraically closed. If $V$ denotes the integral closure of $R$ in its ring of fractions, then the conditions above are equivalent to
\begin{enumerate}[(i)]
\setcounter{enumi}{4}
\item \label{thm5} $\Tor_i^R(M,V) = 0$.
\end{enumerate}
\end{theorem}
\begin{proof} 
We will show that (\ref{thm1}) $\Rightarrow$ (\ref{thm2}) $\Rightarrow$ (\ref{thm3}) $\Rightarrow$ (\ref{thm4}) $\Rightarrow$ (\ref{thm1}). We assume (\ref{thm1}). Let $\p \in \MIN(R)$. 
Since $M$ has finite length we have $M_\p = 0$, and thus
\[
\Tor_j^R(M,\eee{(R/\p)})_\p =  \Tor_j^{R_\p}(M_\p,\eee{(R/\p)}_\p)=0
\]
for all $j \gs 0$. In particular, the complex 
\[
\xymatrixcolsep{5mm}
\xymatrixrowsep{2mm}
\xymatrix{
(G_\bullet \otimes \eee(R))_\p: \ldots \ar[r] & (G_{i+1})_\p \ar[rr]^-{(\varphi_{i+1}^{[q]})_\p} && (G_i)_\p \ar[rr]^-{(\varphi_i^{[q]})_\p} && (G_{i-1})_\p \ar[r] &  \ldots \ar[r] & (G_0)_\p \ar[r] & 0
}
\]
is split exact. 
All the entries in a matrix associated to $\varphi_{i+1}$ are in $H^0_\m(R)$, and in particular they are nilpotent.
We choose $q_0=p^{e_0}$ such that $\IM(\varphi_{i+1}^{[q]}) = 0$ for all $q \gs q_0$. For such $q$, we have $(\varphi_{i+1}^{[q]})_\p \equiv 0;$ therefore,
$(G_i)_\p$ splits inside $(G_{i-1})_\p$ via $(\varphi_i^{[q]})_\p$. This means that
\begin{equation}
\label{BE}
b_i := \rk((G_i)_\p) = \rk(G_i) = \rk((\varphi_{i}^{[q]})_\p) \ \ \ \ \ \mbox{ and } \ \ \ \ \ I_{b_i}(\varphi_{i}^{[q]}) \not\subseteq \p,
\end{equation}
where $I_{r}(\psi)$ denotes the Fitting ideal of an homomorphism $\psi:G \to H$ of rank $r$ between two free modules $G$ and $H$. Note that localizing and taking powers only decreases the rank of $\varphi_i$, and $b_i$ is already the maximal possible rank. Thus $b_i = \rk(\varphi_i^{[q]})$ for all $q \gs 1$. Furthermore, if $I_{b_i}(\varphi_{i})$ was contained in $\p$, then so would be $I_{b_i}(\varphi_{i}^{[q]})$. Hence (\ref{BE}) holds in fact for all $q =p^e$. Consider the following complex
\[
\xymatrixcolsep{5mm}
\xymatrixrowsep{2mm}
\xymatrix{
0 \ar[r] & G_i \otimes R/\p \ar[rrr]^-{\varphi_i^{[q]} \otimes 1_{R/\p}} &&& G_{i-1} \otimes R/\p \ar[r] & C_q \ar[r] & 0,
}
\]
where $C_q$ is the cokernel. By Buchsbaum-Eisenbud Theorem \cite{BEComplexExact}, the two conditions (\ref{BE}) ensure that it is acyclic for all $q$. Then, 
\[
\Tor_i^R(M,\eee{(R/\p)}) = \Tor_1^R(C_q,R/\p) = 0.
\]
for all $e \gs 0$.
This holds for all $\p \in \MIN(R)$, proving (\ref{thm2}). 

Clearly (\ref{thm2}) implies (\ref{thm3}). We now show (\ref{thm3}) $\Rightarrow$ (\ref{thm4}). Since for all $\p \in \MIN(R)$, we have $\Tor_i^R(M,\eee {(R/\p)}) = 0$ for $e \gg 0$, in particular $\beta_i^F(M,R/\p) = 0$. Hence
\[
\ds \beta_i^F(M,R) = \sum_{\p \in \MIN(R)} \left[\beta_i^F(M,R/\p) \lambda_{R_\p}(R_\p)\right] = 0.
\]

We now prove (\ref{thm4}) $\Rightarrow$ (\ref{thm1}).
Now suppose that $\beta_i^F(M,R) = 0$. By Lemma \ref{positive}, we have 
\[
\ds \IM(\varphi_{i+1}) \subseteq \bigcap_{\p \in \MIN(R)} \p G_i  = \sqrt{0}G_i.
\]
Since the image is nilpotent, as noticed above in (\ref{BE}) taking $q=1$,  we have
\[
b_i = \rk(G_i) = \rk(\varphi_{i}) \ \ \ \ \ \mbox{ and } \ \ \ \ \ I_{b_i}(\varphi_{i}) \not\subseteq \p
\]
for all $\p \in \MIN(R)$. Localizing the resolution at any $\p \in \MIN(R)$ gives a split exact complex
\[
\xymatrixcolsep{5mm}
\xymatrixrowsep{2mm}
\xymatrix{
(G_\bullet)_\p: 0 \ar[r] & (G_i)_\p \ar[rr]^-{(\varphi_i)_\p} && \ldots \ldots\ar[r] & (G_0)_\p \ar[r] & 0.
}
\]
In particular, $\IM((\varphi_{i+1})_\p) = (\IM(\varphi_{i+1}))_\p =0$. This holds for all minimal primes $\p$ of $R$, proving that $\IM(\varphi_{i+1}) \subseteq H^0_\m(G_i)$.

Finally, assume that $R$ is complete and $K$ is algebraically closed, and let $V$ be integral closure of $R$ in its ring of fractions. Let $\p \in \MIN(R)$ and let $V(\p)$ be the integral closure of $R/\p$, which is a DVR. By Lemma \ref{qthPowerParameter}, we have that $\eee (R/\p) \cong \bigoplus V(\p)$ for all $e \gg 0$. Condition (\ref{thm3}) implies that $\Tor_i^R(M,\eee {(R/\p)}) \cong \bigoplus \Tor_i^R(M,V(\p)) = 0$, therefore $\Tor_i^R(M,V(\p))= 0$ for all $\p \in \MIN(R)$. Since $V \cong \bigoplus_{\p \in \MIN(R)} V(\p)$, we see that (\ref{thm3}) implies (\ref{thm5}). Conversely, if $\Tor_i^R(M,V) =0$, by the same argument, we get that $\Tor_i^R(M,V(\p)) = 0$ implies $\Tor_i^R(M,\eee {(R/\p)}) = 0$ for all $e \gg 0$ and for all $\p \in \MIN(R)$. Then, (\ref{thm5}) implies (\ref{thm3}).
\end{proof}

\begin{corollary}\label{CorProjCMDimOne}
Let $(R,\m,K)$ be a one-dimensional Cohen-Macaulay local ring of characteristic $p>0$, and $M$ be an $R$-module of finite length. Then the following are equivalent:
\begin{enumerate}[(i)]
\item \label{Cor1} $\beta_i^F(M,R) = 0$ for all $i \gs 1$.
\item \label{Cor2} $\beta_i^F(M,R) = 0$ for some $i \gs 1$.
\item \label{Cor3} $\pd_R(M) < \infty$.
\end{enumerate} 
\end{corollary}
\begin{proof} 
Clearly (\ref{Cor1}) implies (\ref{Cor2}). Now assume (\ref{Cor2}), we want to show that (\ref{Cor3}) holds. By assumption, there exists an integer $i \gs 1$ such that $\beta_i^F(M,R) = 0$. Then, Theorem \ref{equiv} implies that $\IM(\varphi_{i+1}) \subseteq H^0_\m(G_i)$, where $(G_j,\varphi_j)_{j \gs 0}$ is a minimal free resolution of $M$. However, $R$ has positive depth and hence
\[
\IM(\varphi_{i+1}) = H^0_\m(\IM(\varphi_{i+1})) \subseteq H^0_\m(G_i) = 0,
\]
since $G_i$ is a free module. Thus $\IM(\varphi_{i+1}) = 0$ and  $\pd_R(M) < \infty$. Finally, if (\ref{Cor3}) holds, we have $\Tor_i^R(M,\eee R) = 0$ for all $i \gs 1$ and $e \gs 0$ \cite[Th\'eor\`eme 1.7]{P-S}. In particular, $\beta_i^F(M,\eee R) = 0$ for all $i \gs 1$.
\end{proof}

\begin{corollary}\label{CorPrjDimGralOneDim}
Let $(R,\m,K)$ be a one-dimensional local ring of characteristic $p>0,$ and let $M$ be a finite length $R$-module. If $\beta_i^F(M,R) = \beta_{i+1}^F(M,R) = 0$ for some $i \gs 1$, then $\pd_R(M)< \infty$. In particular, for any parameter $x$, if $\beta_2^F(R/(x),R) = 0$ then $R$ is Cohen-Macaulay.
\end{corollary}
\begin{proof} 
Let $(G_j,\varphi_{j})_{j \gs 0}$ be a minimal free resolution of $M$. Since $\beta_i^F(M,R) = 0$, we have that $\IM(\varphi_{i+1})$ has finite length, and it is nilpotent. Take $q = p^e\gg 0$ so that $\IM(\varphi_{i+1}^{[q]}) = 0$. For such $q$ we  have $\Ker(\varphi_{i+1}) = G_{i+1}$. Since the resolution is minimal, we get
\[
\ds \lambda(\Tor_{i+1}^R(M,\eee R)) = q^\alpha\lambda\left(\frac{G_{i+1}}{\IM(\varphi_{i+2}^{[q]})}\right) \gs q^\alpha\lambda\left(\frac{R}{\m^{[q]}}\right),
\]
where the last inequality comes from projecting onto one of the components of $G_{i+1}$. Dividing by $q$ and taking limits, we get
\[
\ds \beta_{i+1}^F(M,R) = \lim_{e \to \infty} \frac{\lambda(\Tor_{i+1}^R(M,\eee R))}{q^{(1+\alpha)}} \gs \lim_{e \to \infty} \frac{\lambda(R/\m^{[q]})}{q} = \ee_{HK}(\m,R) > 0,
\]
which is a contradiction. 

The last claim follows from the fact that for any parameter $x$, we have
\[
\beta_1^F(R/(x),R) \ls \lim_{e \to \infty} \frac{\lambda(H_1(x^q;R))}{q} = 0,
\]
where $H_1$ denotes the first Koszul homology (see \cite{RobertsNIT} and \cite[Theorem 6.2]{HoHuPhantom}).
\end{proof}
\begin{lemma} \label{fpdim} 
Let $(R,\m,K)$ be a  local ring of positive characteristic $p>0,$ and  $\p \in \Spec(R)$. If $\pd_R(\p) < \infty$, then $R$ is a domain.
\end{lemma}
\begin{proof} Since $\p$ has finite projective dimension, given a minimal free resolution 
\[
\xymatrixcolsep{5mm}
\xymatrixrowsep{2mm}
\xymatrix{
0 \ar[r] & L_t  \ar[rr]^-{\psi_t} &&  \ldots \ar[r] & L_0 \ar[r] & R/\p \ar[r] & 0
}
\]
of $R/\p$ over $R$, we have that
\[
\xymatrixcolsep{5mm}
\xymatrixrowsep{2mm}
\xymatrix{
0 \ar[r] & L_t  \ar[rr]^-{\psi_t^{[q]}} && \ldots \ar[r] & L_0 \ar[r] & R/\p^{[q]}\ar[r] & 0
}
\]
is a minimal free resolution of $R/\p^{[q]}$ 
over $R$ \cite[Exemples 1.3 d)]{P-S}. Then $\Ass_R(R/\p^{[q]}) = \{\p\}$, and so,
$\p^{[q]}$ is $\p$-primary for all $q = p^e$. Let $x \notin \p$ and assume $xy = 0$ for $y \in R$. This implies that for any  $q$, we have $xy \in \p^{[q]}$. We conclude that $y \in \p^{[q]}$ since $x \notin \p$. Thus 
\[
y \in \bigcap_{q \gs 1} \p^{[q]} = (0).
\]
In particular, the localization map $R \to R_\p$ is injective.  We have that $\pd_R(R/\p) < \infty$ implies $\pd_{R_\p}(k(\p)) < \infty$. Then,  $R_\p$ is a regular local ring; in particular, a domain.
Therefore, $R$ is a domain.
\end{proof}

\begin{proposition} 
Let $(R,\m,K)$ be a one-dimensional local ring of characteristic $p>0,$ and  $I$ be an $\m$-primary integrally closed ideal. If $\beta_i^F(R/I,R) = 0$ for some $i > 0$, then $R$ is regular.
\end{proposition}
\begin{proof} 
Let $\p$ be a minimal prime of $R$. Since $\beta_i^F(R/I,R) = 0$, by Theorem \ref{equiv} we have that $\Tor_i^R(R/I,R/\p) = 0$.  By Proposition \ref{int_closed}, it follows that $\pd_R(R/\p) < \infty$, and thus $R$ is a domain by Lemma \ref{fpdim}. Since one-dimensional local domains are Cohen-Macaulay, by Corollary \ref{CorProjCMDimOne} we have that $\pd_R(R/I)<\infty$. In particular $\Tor_j^R(R/I,K) = 0$ for $j\gg0$. We conclude that $\pd_R(K) < \infty$ because $R/I$ tests finite projective dimension \cite[Theorem 5(ii)]{Burch}. Hence, $R$ is regular.
\end{proof}

\section{Syzygies of finite length}\label{SecSyzFinLen}

We now present several characteristic-free results. In particular, we do not always assume that the rings have positive characteristic. We focus on Question \ref{Q SyzFinLen}. Specifically, we give support to the claim that a finite length $R$-module $M$ of infinite projective dimension cannot have a finite length syzygy $\Omega_i$ for $i > \dim(R)+1$. As a consequence of our methods, we describe, in some cases, the dimension of the syzygies.

It follows from Theorem \ref{equiv} that if $\dim (R) = 1$ and $R$ has positive characteristic, then a positive answer to Question \ref{Q SyzFinLen}  is equivalent to the statement:
for every  $M$   of finite length, $\beta_i^F(M,R) = 0$ for some  $i>1$ implies  $\pd_R(M) < \infty.$

We now provide an example that shows that the requirement of $i>\dim(R)+1$ in Question \ref{Q SyzFinLen} is necessary to have a positive answer. 

\begin{example}\label{finitelength}
Let $R = \mathbb{F}_p\ps{x,y}/(x^2,xy)$ and $M = R/(x)$. Then $\dim(R) = 1$. In addition, $\pd_R(M) = \infty$ because $R$ is not Cohen-Macaulay. We have that  $\Omega_2 \cong H^0_{(x,y)}(R) = (x)$ has finite length.
\end{example}

\begin{lemma} \label{lemmah0} Let $(R,\m,K)$ be a local ring and let $M$ be a finite length $R$-module that has a finite length syzygy $\Omega_{i+1}$, for some fixed $i > 0$. Then,
\[
\Tor_i^R\left(M,R/H^0_\m(R)\right) = 0.
\]
If $R$ has positive characteristic $p$, then for all $e \gs 0$
\[
\Tor_i^R\left(M,\eee{\left(R/H^0_\m(R)\right)}\right) = 0.
\]
\end{lemma}
\begin{proof}
Set $H:=H^0_\m(R)$. Let $(G_\bullet,\varphi_\bullet)$ be a minimal free resolution of $M$:
\[
\xymatrixcolsep{5mm}
\xymatrixrowsep{2mm}
\xymatrix{
G_\bullet: \ldots \ar[r] &G_{i+1} \ar[rr]^-{\varphi_{i+1}} &&  G_i \ar[rr]^-{\varphi_i} &&  G_{i-1} \ar[rr]^-{\varphi_{i-1}} && G_{i-2} \ar[r] &\ldots\ar[r]& G_0 \ar[r] & M \ar[r] & 0.
}
\]
Tensor $G_\bullet$ with $R/H$ and denote by $\ov{G}_\bullet$ its residue class modulo $H$:
\[
\xymatrixcolsep{5mm}
\xymatrixrowsep{2mm}
\xymatrix{
\ov{G}_{i+1} \ar[rr]^-{\ov{\varphi}_{i+1}} &&  \ov{G}_i \ar[rr]^-{\ov{\varphi}_i} &&  \ov{G}_{i-1}
}
\]
Since $\IM(\varphi_{i+1}) = \Omega_{i+1}$ has finite length by assumption, we have $\ov{\varphi}_{i+1} = 0$. We want to show that $\Ker(\ov{\varphi}_{i}) = 0 $ as well. For any $\p \in \Spec(R) \smallsetminus \{\m\}$, the complex $(G_\bullet)_\p$ is split exact:
\[
\xymatrixcolsep{5mm}
\xymatrixrowsep{2mm}
\xymatrix{
0\ar[rr] &&  (G_i)_\p  \ar[rr]^-{(\varphi_i)_\p} &&  (G_{i-1})_\p \ar[rr]^-{(\varphi_{i-1})_\p} && (G_{i-2})_\p \ar[r] &\ldots  \ar[r]& (G_0)_\p \ar[rr] &&  0,
}
\]
because $M$ and $\Omega_{i+1}$ have finite length. 
We have that $\rk((\varphi_i)_\p)$ is maximal, because $\rk(G_i) \ls \rk(G_{i-1})$ as the localized complex is split exact, and localizing only decreases the rank of a map. 
Thus, $r:=\rk(G_i) = \rk((\varphi_i)_\p) = \rk(\varphi_i)$.
Furthermore, $I_r(\varphi_i) \not\subseteq \p$, by split exactness. Since this holds for all $\p \in \Spec(R) \smallsetminus\{\m\};$ in particular, we have $\Depth(I_r(\ov{\varphi}_i)) \gs 1$. By the Buchsbaum-Eisenbud Criterion, we have that
\[
\xymatrixcolsep{5mm}
\xymatrixrowsep{2mm}
\xymatrix{
0 \ar[r] &  \ov{G}_i \ar[rr]^-{\ov{\varphi}_i} &&  \ov{G}_{i-1} \ar[rr] && \ov{\Omega}_{i-1} = \Omega_{i-1}/H\Omega_{i-1} \ar[r] & 0.
}
\]
is an exact complex.
Therefore $\Ker(\ov{\varphi}_i) = 0$, and hence $\Tor_i^R(M,R/H) = 0$. For the second part of the Lemma, when $R$ has positive characteristic, the argument is the same: just notice that the complex $\eee(G_\bullet)_\p$ is again split exact for all primes $\p \ne \m$ and apply the same argument as above to the map $\ov{\varphi}_i^{[q]}$.
\end{proof}

We now give  results that support an affirmative answer to Question \ref{Q SyzFinLen} for one-dimensional rings. Over Buchsbaum rings, the modules $H^i_\m(R)$ are $K$-vector spaces for $i < \dim(R)$. Because of this fact, we can prove the following proposition using Lemma \ref{lemmah0}.

\begin{proposition}\label{Prop Buchsbaum}
Let $(R,\m,K)$ be a one-dimensional Buchsbaum ring. Then the answer to Question \ref{Q SyzFinLen} is positive.
\end{proposition}
\begin{proof}
Assume that there exists a finite length $R$-module $M$ such that $\Omega_{i+1}(M)$ has finite length for some $i \gs 2$. By Lemma \ref{lemmah0}, we have
\[
0 = \Tor_i^R\left(M,R/H^0_\m(R)\right) \cong \Tor_{i-1}^R\left(M,H^0_\m(R)\right),
\]
where $i-1 \gs 1$ for dimension shifting. By Remark \ref{Rem Buchsbaum}, we have that $H^0_\m(R) \cong \bigoplus_{j=1}^t K$. Therefore,
\[
0 = \Tor_{i-1}^R\left(M,H^0_\m(R)\right) = \bigoplus_{j=1}^t \Tor_{i-1}^R\left(M,K\right),
\]
which implies $\Tor_{i-1}^R(M,K) = 0$. Hence, $\pd_R(M) \ls i-2$.
\end{proof}

We now present two results about the dimension of syzygies of a finite-length module. These results will be used in Proposition \ref{Prop Betti non decreasing} to give a case in which a finite-length module cannot have infinitely many syzygies of finite length.

\begin{proposition} \label{dimsyz} Let $(R,\m,K)$ be a local ring of dimension $d$ and let $M$ be a finite length $R$-module. Let $i \gs 1$ and let $\Omega_i$ be the $i$-th syzygy of $M$. Then either $\dim(\Omega_i) = d$ or $\Omega_i$ has finite length.
\end{proposition}
\begin{proof}
By way of contradiction, we suppose $\dim(\Omega_i) = k$ with $0 < k < d$. 
Let $G_\bullet \to M \to 0$ be a minimal free resolution of $M$. 
By our assumption on $\dim(\Omega_i),$ we can choose $\p \in \MIN(\ann(\Omega_i)) \smallsetminus \left(\{\m\} \cup \MIN(R)\right)$ and localize $G_\bullet$ at $\p$. The resulting complex is split exact, because $M_\p = 0$. In particular, $(\Omega_i)_\p$ is a free $R_\p$-module. By our choice of $\p$, we have that $(\Omega_i)_\p$ has finite length, and $\dim(R_\p) > 0$, a contradiction.
\end{proof}

\begin{proposition} \label{one-dim} 
Let $(R,\m,K)$ be a local ring of positive dimension. Suppose there exists an $R$-module $M$ of infinite projective dimension and finite length which has a finite length syzygy $\Omega_{i+1}$, for some fixed $i > 0$. If $\beta_i(M) \gs \beta_{i-1}(M)$, then $\Omega_{i-1}$ has finite length as well and $R$ is one-dimensional.
\end{proposition}
\begin{proof} 
Let $(G_\bullet,\varphi_\bullet)$ be a minimal free resolution of $M$:
\[
\xymatrixcolsep{5mm}
\xymatrixrowsep{2mm}
\xymatrix{
G_\bullet: G_{i+1} \ar@{->>}[dr] \ar[rr]^-{\varphi_{i+1}} &&  R^{\beta_i(M)} \ar@{->>}[dr] \ar[rr]^-{\varphi_i} &&  R^{\beta_{i-1}(M)} \ar@{->>}[dr] \ar[rr]^-{\varphi_{i-1}} && G_{i-2} \ar[r] &\ldots  \ar[r]& G_0 \ar[r] & M \ar[r] & 0. \\
& \Omega_{i+1} \ar@{^{(}->}[ru]  && \Omega_{i} \ar@{^{(}->}[ru]   && \Omega_{i-1} \ar@{^{(}->}[ru]
}
\]
Let $\p \in \Spec(R) \smallsetminus\{\m\}$. We localize $G_\bullet$ at $\p$. Since both $M$ and $\Omega_{i+1}$ have finite length, we have a split exact sequence
\[
\xymatrixcolsep{5mm}
\xymatrixrowsep{2mm}
\xymatrix{
G_\bullet: 0\ar[rr] &&  R^{\beta_i(M)}_\p \ar[dr]_-{\cong}   \ar[rr] &&  R^{\beta_{i-1}(M)}_\p \ar@{->>}[dr]  \ar[rr] && (G_{i-2})_\p \ar[r] &\ldots  \ar[r]& (G_0)_\p \ar[rr] &&  0\\
& && (\Omega_{i})_\p \ar@{^{(}->}[ru] && (\Omega_{i-1})_\p.
}
\]
In particular, this implies $\beta_i(M) \ls \beta_{i-1}(M)$. Since the opposite inequality holds  by or assumption, we have  equality. 
Set $\beta = \beta_i(M) = \beta_{i-1}(M)$. 
From the split exact sequence above, we also get that $R_\p^\beta \cong (\Omega_i)_\p$; therefore, $(\Omega_{i-1})_\p = 0$. 
Since $\p$ is an arbitrary prime in $\Spec(R)\smallsetminus\{\m\}$,
we have that $\Omega_{i-1}$ has finite length. 
Thus, we have a free complex $0 \to F_1=R^\beta \to F_0=R^\beta \to 0$ with finite length homology.  We conclude that $R$ has dimension one by the New Intersection Theorem \cite{RobertsNIT}.
\end{proof}

\begin{remark} \label{Rmk Betti numbers} 
If in Proposition \ref{one-dim} one assumes that the sequence of Betti numbers $\{\beta_i(M)\}$ is non-decreasing, then one can repeat the argument above to show that $i$ is necessarily odd, and $\beta_i(M) = \beta_{i-1}(M)$, $\beta_{i-2}(M) = \beta_{i-3}(M)$, $\ldots$, $\beta_1(M) = \beta_0(M)$. In addition, 
 $\Omega_{j}(M)$ has finite length for all $j$ even, $0 \ls j \ls i+1$. In particular, the typical situation to study would be $(R,\m,K)$ a one-dimensional ring and a resolution
\[
\xymatrixcolsep{5mm}
\xymatrixrowsep{2mm}
\xymatrix{
0 \ar[r] &  \Omega_4 \ar[r] & R^\beta \ar[r] &  R^\beta \ar[dr] \ar[rr] && R^\alpha \ar[r] & R^\alpha \ar[r] & M \ar[r] & 0\\
&&&&\Omega_2 \ar[ru]
}
\]
with $\Omega_4$ and $\Omega_2$ of finite length.
\end{remark}

As a consequence of these results, we give a partial answer to Question \ref{Q SyzFinLen} in the case where $M$ has eventually non-decreasing Betti numbers.
It is a conjecture of Avramov that every finitely generated module over a local ring has eventually non-decreasing Betti numbers \cite{AvramovBetti}. The conjecture is know to be true in several cases \cite{EisenbudBettiHyp,Lescot,Choi,Sun1,CIdim,Sun2}, in particular, for Golod rings \cite[Corollaire 6.5]{Lescot}.

\begin{proposition} \label{Prop Betti non decreasing} 
Let $(R,\m,K)$ be a local ring and let $M$ be a finite length $R$-module of infinite projective dimension with eventually non-decreasing Betti numbers. Then, for all $i \gg 0$, there exists $\p \in \MIN(R)$ such that $\dim(\Omega_i) = \dim(R/\p)$. In particular, $M$ cannot have arbitrarily high syzygies of finite length.
\end{proposition} 
\begin{proof} If $\Supp(\Omega_i) \cap \MIN(R) \ne \emptyset$ for all $i \gg 0$ then we are done. By way of contradiction, assume that there exist infinitely many syzygies $\Omega_i$ of $M$ such that $\Supp(\Omega_i) \cap \MIN(R) = \emptyset$. Notice that, by Proposition \ref{dimsyz}, such syzygies must have finite length. By replacing $M$ with a high enough syzygy, we can then assume that $M$ is a module of finite length with non-decreasing Betti numbers, and with infinitely many syzygies of finite length. 
We have that $R$ is one-dimensional by  Proposition \ref{one-dim}. Furthermore, by Remark \ref{Rmk Betti numbers}, we have $\beta_{2i} = \beta_{2i+1}$ for all $i \gs 0$. For $i \gs 0$ consider the short exact sequence
\[
\xymatrixcolsep{5mm}
\xymatrixrowsep{2mm}
\xymatrix{
0 \ar[r] &  \Omega_{2i+2} \ar[r] & R^{\beta} \ar[r]^-{\varphi} &  R^{\beta} \ar[r] & \Omega_{2i} \ar[r] & 0,
}
\]
where $\beta:=\beta_{2i} = \beta_{2i+1}$. Let $S:=R[\varphi]$. Then $R^\beta$ becomes an $S$-module. The exact sequence above shows that $\Omega_{2i} \cong R^\beta \otimes_S S/(\varphi)$ and $\Omega_{2i+2} \cong (0:_{R^\beta} \varphi)$. Then, by \cite[Proposition 11.1.9 (2)]{SwansonHuneke},
\[
\ds \lambda(\Omega_{2i}) - \lambda(\Omega_{2i+2}) = e(\varphi;R^\beta), 
\]
where $e(\varphi;-)$ denotes the Hilbert-Samuel multiplicity with respect to the ideal $(\varphi)$ in $S$. 
Since such multiplicity is always positive, we have that $\lambda(\Omega_{2i+2})<\lambda(\Omega_{2i})$, for all $i\gs 0$. Since there cannot be an infinite strictly decreasing sequence of such lengths, we obtain a contradiction.
\end{proof}

\begin{remark} 
Proposition \ref{Prop Betti non decreasing} also follows from \cite[Theorem 8]{BeckLeamer}, and it gives another proof of the fact that, when $M$ is a module of finite length with eventually non-decreasing Betti numbers and $R$ is equidimensional, then the sequence of integers $\{\dim(\Omega_i)\}_{i=0}^\infty$ is constant for $i \gg 0$ (see \cite[Corollary 2]{BeckLeamer}).
\end{remark}

\begin{proposition} \label{xi} 
Let $(R,\m,K)$ be a one-dimensional local ring. Suppose that there exists a finite length module $M$ of infinite projective dimension that has a finite length syzygy $\Omega_{i+1}$, for some fixed $i \gs 2$. Then
\[
\lambda(\Omega_{i+1}) = \sum_{j = 0}^{i} (-1)^{i-j+1} \lambda(\Tor_j^R(M,R/(x))),
\]
where $x$ is a suitable parameter.
\end{proposition}
\begin{proof}
Consider a minimal free resolution of $M$:
\[
\xymatrixcolsep{5mm}
\xymatrixrowsep{2mm}
\xymatrix{
G_\bullet: & \ldots \ar[r] & G_{i+1} \ar[dr] \ar[rr] && G_{i} \ar[dr] \ar[rr] && G_{i-1}\ar[r]&\ldots \ar[r] & G_1 \ar[dr] \ar[rr] && G_0 \ar[r] & M \ar[r] & 0. \\
&&& \Omega_{i+1} \ar[ru]&& \Omega_{i} \ar[ru]& &&& \Omega_1\ar[ru]
}
\]
For all $j = 1,\ldots,i+1$, we  break it into short exact sequences:
\[
\xymatrixcolsep{5mm}
\xymatrixrowsep{2mm}
\xymatrix{
0 \ar[rr] && \Omega_j \ar[rr] && G_{j-1} \ar[rr] && \Omega_{j-1} \ar[rr] && 0,
}
\]
where $\Omega_0:= M$. These give two exact sequences:
\[
\xymatrixcolsep{5mm}
\xymatrixrowsep{2mm}
\xymatrix{
0 \ar[rr] && \Omega_{i+1} \ar[rr] && H^0_\m(G_{i}) \ar[rr] && H^0_\m(\Omega_{i})\ar[rr] && 0
}
\]
and
\[
\xymatrixcolsep{5mm}
\xymatrixrowsep{2mm}
\xymatrix{
0 \ar[rr] && H^0_\m(\Omega_j) \ar[rr] && H^0_\m(G_{j-1}) \ar[rr] && H^0_\m(\Omega_{j-1}).
}
\]
The first short exact sequence comes from the fact that $\Omega_{i+1}$ has finite length, and so, $H^1_\m(\Omega_{i+1}) = 0$. Furthermore, the cokernel of the rightmost map in the second exact sequence, which can be proved to be the kernel of the leftmost map in
\[
\xymatrixcolsep{5mm}
\xymatrixrowsep{2mm}
\xymatrix{
\Omega_j \otimes_R H^1_\m(R) \ar[rr] && G_{j-1} \otimes_R H^1_\m(R) \ar[rr] && \Omega_{j-1} \otimes_R H^1_\m(R) \ar[rr] && 0
}
\]
is then $\Tor_1^R(\Omega_{j-1},H^1_\m(R))$. For simplicity, we denote $\omega_j := \lambda(H^0_\m(\Omega_j))$, $g_j:= \lambda(H^0_\m(G_j))$ and $\alpha_j:= \lambda(\Tor_1^R(\Omega_j,H^1(R)))$. Then, we have relations
\begin{eqnarray*}
\begin{array}{ll}
\omega_{i+1} &= g_{i} - \omega_{i} \\
\omega_{i} &= g_{i-1} - \omega_{i-1} + \alpha_{i-1} \\
&\vdots \\
\omega_2 &= g_1 - \omega_1 + \alpha_1 \\
\omega_1 &= g_0 - \lambda(M) + \lambda(\Tor_1(M,H^1_\m(R))).
\end{array}
\end{eqnarray*}

After localizing the  resolution $G_\bullet$ at any minimal prime $\p$, since $(\Omega_{i+1})_\p = 0$, we obtain that  $\sum_{j=0}^{i} (-1)^j \beta_j(M) = 0$.
Then, $\sum_{j=0}^{i} (-1)^j g_j = 0$ because $g_j=\beta_j(M)\cdot \lambda(H^0_m(R)).$
Therefore,
\[
\omega_{i+1} = \lambda(\Omega_{i+1}) = \sum_{j=1}^{i-1} (-1)^{i-j} \alpha_j + (-1)^{i} \lambda(\Tor_1(M,H^1(R))) +  (-1)^{i-1} \lambda(M).
\]

Choose a parameter $x$ such that $H^0_\m(R) = 0:_Rx$, as in Remark \ref{RemParameter}. By similar considerations we can also assume that $xM=0$. From this choice, we have that  $xH^0_\m(\Omega_j) = 0$ for all $j = 0,\ldots,i+1,$ because $\Omega_j \subseteq G_{j-1}$ is a free $R$-module. 
Since the Tor modules can be computed using flat resolutions,
we have an exact sequence
\[
\xymatrixcolsep{5mm}
\xymatrixrowsep{2mm}
\xymatrix{
0 \ar[r] & H^0_\m(R) \ar[r] & R \ar[r] & R_x \ar[r] & H^1_\m(R) \ar[r] & 0.
}
\]
We complete on the left to get a flat resolution of $H^1_\m(R)$:
\[
\xymatrixcolsep{5mm}
\xymatrixrowsep{2mm}
\xymatrix{
\ldots \ar[r] & R^{\mu(H^0_\m(R))} \ar[rr] && R \ar[rd]\ar[rr]&& R_x \ar[r] & H^1_\m(R) \ar[r] & 0 \\
&&&&R/H^0_\m(R) \ar[ru]
}
\]
By our  choice of $x$, we have that a free resolution of $R/x$ starts as
\[
\xymatrixcolsep{5mm}
\xymatrixrowsep{2mm}
\xymatrix{
\ldots \ar[r] & R^{\mu(H^0_\m(R))} \ar[r] & R \ar[r] & R \ar[r] & R/(x) \ar[r] & 0.
}
\]
For all $j = 1,\ldots, i-1$, we obtain
\[
\Tor_1^R(\Omega_{j},H^1_\m(R)) \cong \Tor_1^R(\Omega_{j},R/(x)) \cong \Tor_{j+1}^R(M,R/(x)),
\]
where the last isomorphism comes from dimension shifting. In addition,
\[
\Tor_1^R(M,H^1_\m(R)) \cong \Tor_1^R(M,R/(x)).
\]
Finally, since $x H^0_\m(\Omega_0) = xM = 0$, we get
\[
M \cong M/xM \cong \Tor_0^R(M,R/(x)),
\]
and the proposition then follows.
\end{proof}

\begin{corollary}\label{Cor 1-3}
Let $(R,\m,K)$ be a one-dimensional  ring, and let $M$ be a finite length module of infinite projective dimension. Then $\lambda(\Omega_1) = \lambda(\Omega_3) = \infty$.
\end{corollary}
\begin{proof} 
Note that $\lambda(\Omega_1) = \infty$ ; otherwise, we would have a short exact sequence
\[
\xymatrixcolsep{5mm}
\xymatrixrowsep{2mm}
\xymatrix{
0 \ar[r] &  \Omega_1 \ar[r]&  G_0  \ar[r]&  M \ar[r] & 0
}
\]
in which both $\Omega_1$ and $M$ have finite length. This cannot happen because $G_0 \ne 0$ is free and $\dim(R) = 1$. Now let us assume by way of contradiction that $\lambda(\Omega_3) < \infty$. Let $(G_\bullet,\varphi_\bullet)$ be a minimal free resolution of $M$:
\[
\xymatrixcolsep{5mm}
\xymatrixrowsep{2mm}
\xymatrix{
0 \ar[r] &  \Omega_3 \ar[rr]& &  G_2  \ar[rr]^-{\varphi_2} &&  G_1 \ar[rr]^-{\varphi_1} && G_0 \ar[rr] && M \ar[r] & 0
}
\]
Let $x \in R$ be a parameter such that $xM=x H^0_\m(R) = 0$. Consider the short exact sequence
\[
\xymatrixcolsep{5mm}
\xymatrixrowsep{2mm}
\xymatrix{
0 \ar[r] &  (x) \ar[r] & R \ar[r] &  R/(x) \ar[r] & 0.
}
\]
By our choice of $x$ we have $0:_R x = H^0_\m(R)$, hence $(x) \cong R/H^0_\m(R)$.  After tensoring the sequence with $M$, we obtain that
\[
\xymatrixcolsep{5mm}
\xymatrixrowsep{2mm}
\xymatrix{
0 \ar[r] &  \Tor_1^R(M,R/(x)) \ar[r] & M/H^0_\m(R)M \ar[r] & M \ar[r] &  M/xM \ar[r] & 0.
}
\]
Since $xM = 0$, we obtain 
\[
\ds \lambda(\Tor_1^R(M,R/(x))) = \lambda(M/H^0_\m(R)M).
\]
Then, by Proposition \ref{xi} we have
\begin{align*}
\lambda(\Omega_3) &=  -\lambda(\Tor_2^R(M,R/(x))) + \lambda(\Tor_1^R(M,R/(x))) - \lambda(M) \\
& \ls \lambda(\Tor_1^R(M,R/(x))) - \lambda(M) \\
& = \lambda(M/H^0_\m(R)M) - \lambda(M) \ls 0,
\end{align*}
which gives a contradiction since $\Omega_3 \ne 0$, because $M$ has infinite projective dimension.
\end{proof}
The following example is due to the second author, and it is taken from \cite{BeckLeamer}. It shows the assumption that $M$ has finite length is needed in Corollary \ref{Cor 1-3}.
\begin{example} \label{Ex 1-3 finite length} 
Let $S=\QQ[x, y, z, u, v]$ and let $I \subseteq S$ be the ideal
\[
\ds I = (x^2, xz, z^2, xu, zv, u^2, v^2, zu + xv + uv, yu, yv, yx - zu, yz - xv).
\]
Let $R = S/I$, which is a one-dimensional ring of depth $0$. In this case $y$ is a parameter, $0 :_R y = (u, v, z^2)$ and $(y) = 0 :_R (0 :_R y)$. Let $M$ be the cokernel of the rightmost map in the following exact complex
\[
\xymatrixcolsep{5mm}
\xymatrixrowsep{2mm}
\xymatrix{
\ldots \ar[rr] &&  R^3 \ar[rrr]^-{\left[u \ v \ z^2\right]} &&& R \ar[rr]^-y && R \ar[rrr]^-{\left[{\tiny \begin{array}{c} u \\ v \\ z^2 \end{array} }\right]} &&& R^3.
}
\]
Then $M$ is a one-dimensional module with first and third syzygies $\Omega_1 \cong R/(y)$ and $\Omega_3 \cong 0 :_R y$. These are both modules of finite length since $y$ is a parameter. 
\end{example}
\section*{Acknowledgments}
The second author thanks the National Science Foundation for support through Grant DMS-$1259142$.
The third author thanks the National Council of Science and Technology of Mexico (CONACYT) for support through Grant $\#207063$. We thank the referee for helpful comments and suggestions.

\bibliographystyle{alpha}
\bibliography{References}
\vspace{.25cm}
{\footnotesize

\noindent \small \textsc{Department of Mathematics, University of Virginia, Charlottesville, VA  22903} \\ \indent \emph{Email address}:  {\tt ad9fa@virginia.edu} 
\vspace{.25cm}

\noindent \small \textsc{Department of Mathematics, University of Virginia, Charlottesville, VA  22903} \\ 
\indent \emph{Email address}:  {\tt huneke@virginia.edu}  

\vspace{.25cm}
\noindent \small \textsc{Department of Mathematics, University of Virginia, Charlottesville, VA  22903} \\ 
\indent \emph{Email address}:  {\tt lcn8m@virginia.edu}  

}
\end{document}